\documentclass[10pt,a4paper]{amsart}
\usepackage{verbatim}

\usepackage[T1]{fontenc}
\usepackage{graphicx}
\usepackage{enumerate}
\usepackage{amsmath,amsfonts,amssymb}
\usepackage{color}
\usepackage{cite}
\definecolor{citation}{rgb}{0.2,0.58,0.2} 
\definecolor{formula}{rgb}{0.1,0.2,0.6}
\definecolor{url}{rgb}{0.3,0,0.5}
\usepackage{pgf,tikz}
\usepackage{mathrsfs}
\usepackage{fancyhdr}

\usepackage{dsfont}

\newcommand{\reqnomode}{\tagsleft@false}

\usepackage[colorlinks,pdfpagelabels,pdfstartview = FitH,bookmarksopen = true,bookmarksnumbered = true,urlcolor=url,linkcolor = formula,plainpages = false,hypertexnames = false,citecolor = citation] {hyperref}

\def\dx{\,{\rm d}x}

\def \d{\,{\rm d}}
\def \diver{\,{\rm div}}
\def\dist{\,{\rm dist}}
\def\eup{\,{\rm ess\, sup}}
\def\eif{\,{\rm ess\, inf}}

\def\diam{\,{\rm diam}}

\allowdisplaybreaks
\makeatletter
\DeclareRobustCommand*{\bfseries}{%
  \not@math@alphabet\bfseries\mathbf
  \fontseries\bfdefault\selectfont
  \boldmath
}

\DeclareMathOperator*{\osc}{osc}

\makeatother

\newlength{\defbaselineskip}
\newcommand{\setlinespacing}[1]
           {\setlength{\baselineskip}{#1 \defbaselineskip}}
\newcommand{\mint}{\mathop{\int\hskip -1,05em -\, \!\!\!}\nolimits}

\newtheorem{theorem}{Theorem}
\newtheorem{corollary}{Corollary}[section]
\newtheorem{definition}{Definition}
\newtheorem{remark}{Remark}[section]
\newtheorem{lemma}{Lemma}[section]
\newtheorem{proposition}{Proposition}[section]
\newcommand{\N}{\mathbb{N}}
\numberwithin{equation}{section}
\allowdisplaybreaks[4]

\newcommand{\rr}{\varrho}
\newcommand{\snr}[1]{\lvert #1\rvert}
\newcommand{\nr}[1]{\lVert #1 \rVert}

\newcommand{\vp}{\varphi}

\newcommand{\cJp}{{\mathcal{J}_{\theta,\vp(\cdot)}}}
\newcommand{\ainc}{{{\rm (aInc)}_p}}
\newcommand{\adec}{{{\rm (aDec)}_q}}
\newcommand{\rp}{{[0,\infty)}}
\newcommand{\rn}{{\mathbb{R}^{n}}}

\newcommand{\data}{\textit{\texttt{data}}}

\title[%{\color{magenta} ??? Obstacle problem and ??? }
Removable sets in elliptic equations with Musielak-Orlicz growth]{Removable sets in elliptic equations\\ with Musielak-Orlicz growth
}
\author{Iwona Chlebicka} \address{Iwona Chlebicka\\Faculty of Mathematics, Informatics and Mechanics, University of Warsaw\\ul. Banacha 2, 02-097 Warsaw, Poland} \email{\texttt{iskrzypczak@mimuw.edu.pl}}
\author{Arttu Karppinen}  \address{Arttu Karppinen\\Department of Mathematics and Statistics, FI-20014 University of Turku, Finland} \email{\texttt{arttu.a.karppinen@utu.fi}}

\begin{document}

\subjclass[2010]{35J60, 35J70\vspace{1mm}} %%ALERT CHECK 35J60 23J70 35B65 35D40

\keywords{Measure data problems, Obstacle problem, Potential estimates, Removable sets\vspace{1mm}}

\thanks{{\it Acknowledgements.}\ I. Chlebicka is supported by NCN grant no. 2016/23/D/ST1/01072. A. Karppinen was partly supported by Turku University Foundation. Part of the research was done while A. Karppinen was visiting University of Warsaw. 
\vspace{1mm}}

\begin{abstract}
 We characterize, in the terms of intrinsic Hausdorff measures, the size of~removable sets for H\"older continuous solutions to elliptic equations with Musielak-Orlicz growth. In the general case we provide a result in the new scale that is more relevant and captures -- as special cases -- the classical results, slightly refines the ones provided for problems stated in the variable exponent and double phase spaces and essentially improves the known one in the Orlicz case. 
\end{abstract}
\vspace{3mm}
%{\small \tableofcontents}

\setlinespacing{1.08}
 
\maketitle
\section{Introduction}

\subsection*{Objectives.} We analyze the fine properties of solutions to quasilinear elliptic equations of~a~form
\begin{flalign}\label{A0}
-\diver\, A(x,Du)=0 \ \ \mbox{in} \ \ \Omega,
\end{flalign} where $\Omega\subset\mathbb{R}^n$, $n\ge 2$ is   open and bounded, while the operator satisfies nonstandard growth and coercivity conditions expressed by the means of an inhomogeneous function $\vp:\Omega\times[0,\infty)\to[0,\infty)$ within the framework presented in Section~\ref{ssec:gen-or}. Such conditions place the energy solution to~\eqref{A0} in the Musielak-Orlicz-Sobolev space $W^{1,\vp(\cdot)}(\Omega)$ defined in Section~\ref{ssec:setting} and embrace a natural scope of variable exponent, Orlicz, double phase spaces and their various combinations. Recently, employing this framework became a well-settled stream in nonlinear analysis, see a survey~\cite{IC-pocket}. In turn, we admit in~\eqref{A0} not only (weighted versions of) Laplacian, $p$-Laplacian, $p(x)$-Laplacian, but also their Orlicz and double phase counterparts. Calling solutions to~\eqref{A0} when $-\diver\, A(x,Du)=-\Delta u$ harmonic functions and their natural generalization when $A(x,z)\cdot z\sim |z|^p$ (with the celebrated case of $p$-Laplacian $-\Delta_p u=
-\diver(\snr{Du}^{p-2}Du)$) $p$-harmonic functions, we say that we examine $\mathcal{A}_{\vp(\cdot)}$-harmonic maps, where  the operator   $\mathcal{A}_{\vp(\cdot)}$   defined on $W^{1,\vp(\cdot)}(\Omega)$ acting as
\begin{flalign}\label{calAH}
\langle\mathcal{A}_{\vp(\cdot)}v,w\rangle:=\int_{\Omega}A(x,Dv)\cdot Dw \ \dx\quad \text{for}\quad w\in C^{\infty}_{c}(\Omega).
\end{flalign}

There are various formalisms to describe Musielak-Orlicz spaces as a setting for partial differential equations. We employ here the one provided in~\cite{hahab}, but we refer to~\cite{IC-b} for another possibility. The main features of the spaces are inhomogeneity (space--dependence) initially investigated in the context of the Lavrentiev phenomenon and general growth introduced for the elasticity theory. Let us refer to a selection of very recent results falling into the scope of the existence and regularity theory in this setting to stress the attention the branch enjoys~\cite{bacomi-st,IC-gradest,IC-lower,IC-measure,comi,deoh,depq,demi,gszg,hahakl,ka,kale,Ok}.

 %Nonlinear version of the classical potential theory has been developed by several authors \cite{hekima,kima,kumi,lima,ma,mik,tu}.

{\em Removability. } Consider a harmonic function defined on a subdomain of $\Omega$, that is on $\Omega\setminus E$ for some measurable $E$. The set $E$ is classically called removable if the harmonic function has a~continuous extension which is harmonic in the whole domain, cf.~\cite{ca}. Here we call  a set $E$ removable for H\"older continuous solutions to~\eqref{A0} if an  $\mathcal{A}_{\vp(\cdot)}$-harmonic function in $\Omega \setminus E$ has a H\"older continuous  $\mathcal{A}_{\vp(\cdot)}$-harmonic extension to $\Omega$. The size of removable sets for $p$-harmonic functions is studied sharply in~\cite{kizo}. Namely, having a~relatively closed subset $E\subset \Omega$  with $s$-dimensional Hausdorff measure zero (for $s>n-p$) and $u$  $p$-harmonic in $\Omega\setminus E$, it is proven for what $s$ the extension $\tilde{u}$ is $p$-harmonic in the whole $\Omega$. For the already classical results we refer to \cite{ca,koma,kizo} and to \cite{ono,hi} for problems involving also lower-order terms. In the nonstandard growth framework, the problem of removability has been studied in the case of~variable exponent spaces in \cite{fush,lat}, Orlicz spaces in \cite{chaly} and  double phase spaces in \cite{ChDF}. Following the last mentioned contribution we employ the  intrinsic capacities and the intrinsic Hausdorff measures introduced recently in \cite{bahaha} and \cite{demi}, respectively.  It is commented below how our accomplishment embraces and deepens the known ones.

\subsection*{Main result} Throughout the paper $\Omega \subset \rn$, $n\ge 2$, is an open bounded set. Let a vector field $A:\Omega\times\rn\to\rn$ be a  Caratheodory's function, that is 
\begin{flalign}\label{A-Car}\text{ $x\mapsto A(x,\cdot)$ is measurable $\quad$ and $\quad$ $z\mapsto A(\cdot,z)$ is continuous.}
\end{flalign} Assume further that the following growth and coercivity assumptions hold true for almost all $x\in \Omega$ and all  $z\in \mathbb{R}^{n}\setminus \{0\}$:
\begin{flalign}\label{A}
\begin{cases}
\ \snr{A(x,z)} \le c_1\vp\left(x,\snr{z}\right)/|z|,\\
\ c_2 {\vp\left(x,\snr{z} \right)} \le A(x,z)\cdot z%\\
%\ \snr{A(x_{1},z)-A(x_{2},z)}\le L\snr{a(x_{1})-a(x_{2})}\snr{z}^{q-1},
\end{cases}
\end{flalign} 
with absolute constants $c_1,c_2>0$ and some function $\vp:\Omega\times\rp\to\rp$ being measurable with respect to the first variable, convex with respect to the second one and   satisfying natural  non-degeneracy and balance conditions (A0), (A1), (aInc)$_p$ and (aDec)$_q$ with some $1<p\leq q\leq n$, described in detail in Section~\ref{sec:prelim}.  Let us notice that restricting to growth by a power below the dimension is just fixing attention. Indeed, when the operator has quicker growth, the solutions as functions from $W^{1,\vp(\cdot)}(\Omega)$ are H\"older continuous.

 Moreover, let $A$ be monotone in the sense that 
\begin{flalign}\label{A-monotone}
0< \,\langle A(x,z_{1})-A(x,z_{2}),z_{1}-z_{2}\rangle\quad\text{for almost all $x\in \Omega$ and any distinct $z_{1},z_{2}\in \mathbb{R}^{n}$.}
\end{flalign}   We describe the volume of removable sets for H\"older continuous $\mathcal{A}_{\vp(\cdot)}$--harmonic maps in the terms of the intrinsic Hausdorff measures $\mathcal{H}_{\cJp}$ defined in Section~\ref{sec:haus} with the use of function $\cJp$ given by 
\begin{equation}\label{Jphi}
\cJp(B_R(x_0))= R^{-\theta} \int_{B_R(x_0)}\vp(x, R^{\theta-1}) \ \dx \quad  x\in\Omega,\ \ R>0,
\end{equation}
where $\theta\in(0,1]$.

Our main results read as follows.
\begin{theorem}\label{T6}
Suppose $\Omega \subset \mathbb{R}^n$, $n\ge 2$, is a bounded  open set and $A$ satisfies~\eqref{A-Car}--\eqref{A-monotone} with a convex $\Phi$--function $\vp:\Omega\times\rp\to\rp$ %being measurable with respect to the first variable, convex with respect to the second one and
satisfying (A0), (A1), (aInc)$_p$ and (aDec)$_q$ with some $1<p\leq q\leq n$. Let $E\subset \Omega$ be a closed subset and $u\in C(\Omega)\cap W^{1,\vp(\cdot)}(\Omega\setminus E)$ be a continuous solution to \eqref{A0} in $\Omega\setminus E$ such that there exist some $C_u>0$ and $\theta\in (0,1]$
\begin{flalign*}
\snr{u(x_{1})-u(x_{2})}\le C_u\snr{x_{1}-x_{2}}^{\theta}\qquad\text{for all $\ x_{1}\in E$, $\ x_{2}\in \Omega$.}
\end{flalign*}
 If $\mathcal{H}_{\cJp}(E)=0$ with $\cJp$ as in~\eqref{Jphi}, then $u$ is {$\mathcal{A}_{\vp(\cdot)}$--harmonic} in $\Omega$.
\end{theorem}

\begin{corollary} Suppose that $E$ is a closed set in $\Omega$ and $u\in C^{0,\theta}(\Omega)$ with $0<\theta\leq 1$ is {$\mathcal{A}_{\vp(\cdot)}$--harmonic} in $\Omega\setminus E$, where $\mathcal{A}_{\vp(\cdot)}$ is given by~\eqref{calAH} with $A$ and $\vp(\cdot)$ satisfying the assumptions in Theorem \ref{T6}. 
If   $\mathcal{H}_{\cJp}(E)=0$, then $u$ is {$\mathcal{A}_{\vp(\cdot)}$--harmonic} in $\Omega$.
\end{corollary}

\noindent Let us additionally mention that Corollary~\ref{coro:H0} provides that the sets of  finite $\mathcal{H}_{\vp(\cdot)}$--measure are removable for $\mathcal{A}_{\vp(\cdot)}$-harmonic functions.

\subsection*{Special cases}

Let us specialize our result to several known result that we retrieve or extend. We start with the celebrated classical case.

\subsubsection*{$p$-Laplacian. } We have the following sharp conclusion from Theorem~\ref{T6} for solutions to $-\Delta_p u=0$ with $1<p<\infty$, which retrieves the already classical result of~\cite{kizo}.\\
{\em When $(n-p)/(p-1)<\theta\leq 1$, then a closed set $E$ is removable for $\theta$--H\"older continuous $p$-harmonic function if and only if $E$ is of $(n-p+\theta(p-1))$--Hausdorff measure zero.} \\
Again, since sets of $p$-capacity zero are removable for bounded $p$-harmonic functions, the lower bound for admissible $\theta$ is not a restriction.

\subsubsection*{Nonstandard growth operators. }
We shall present here the extentions of results of removability provided for nonstandard growth problems. Note that in any nonstandard growth setting there is a gap between modular form of estimates and the norm ones making the classical tool of H\"older inequality far less useful. Special (power-type) form of variable exponent spaces or double phase spaces enables to pass it by. In previous studies in the Orlicz growth case~\cite{chaly} the authors agreed on loosing some information by the use of rough estimates. Sticking to modular form of the final estimate we improve several existing results by proving the result in the new and far  more relevant scale.

\subsubsection*{$p(x)$-Laplacian. } We take $\vp(x,s)=s^{p(x)}$, where $p: \Omega \to \mathbb{R}$ is a variable exponent, such that $1<p^{-}_{\Omega} \leq p(x)\leq p^{+}_{\Omega} < \infty$ and $p$ satisfies $\log$-H\"older condition (a special case of (A1))%, which is equivalent to boundedness of $r^{p^{-}_{B} -p_B^{+}}$ %IC: I don't see a gain from writing it.
. Under these assumptions we study solutions to
\begin{align*}
0=-\Delta_{p(x)} u=-\diver(\snr{Du}^{p(x)-2}Du).
 \end{align*} Theorem~\ref{T6} provides removability of~$\mathcal{H}_{\cJp}$-Hausdorff measure zero sets with  
\begin{align*}
{\cJp}(B_R(x_0))=\int_{B_R(x_0)}R^{-p(x)+\theta(p(x)-1)}\,\dx.
\end{align*}
The results extend the known results from~\cite{fush,lat}, where the provided measure comes in fact from the easiest bounds from above to ours, expressed by the means of supremum and/or infimum of $p$.
%When we consider a log-H\"older continuous exponent $p:\Omega\to[p_1,p_2]$ with $1<p_-\leq p_+<\infty$, where $p^{+}_{\Omega}= \sup_{x \in \Omega} p(x)$ and $p^{-}_{\Omega}=\inf_{x \in \Omega} p(x)$, solutions to
%\[0=-\Delta_{p(x)} u=-\diver(\snr{Du}^{p(x)-2}Du).\]
%\cite{fush,lat},  
%In variable exponent context the assumption (A1) is called $\log$-H\"older continuity, which is equivalent to the boundedness of $r^{p^{-}_{B} -p_B^{+}}$. If we consider \eqref{Jphi} for $\vp(x,t)=t^{p(x)}$, we see with the use of $\log$-H\"older continuity that
%\begin{align*}
%R^{\theta(p^{-} -1)} \int_{B_R} R^{-p(x)} \ \dx \leq c R^{n-p^{+} + \theta(p^{-}-1)} \leq c R^{\frac{p^{-}(n-p^{+})}{p^{+}} +\theta(p^{-}-1)}.
%\end{align*}
%This is the result found in \cite{fush}. 

\subsubsection*{Double phase growth operators. } Within the framework developed in~\cite{comi}, in \cite{ChDF} removable sets are characterized for solutions to \[0=-\diver\, A(x,Du)=-\diver\left(\omega(x)\big(\snr{Du}^{p-2}+a(x)\snr{Du}^{q-2}\big)Du\right)\] with $1<p\leq q<\infty$, possibly vanishing weight $0\leq a\in C^{0,\alpha}(\Omega)$ and $q/p\leq 1+\alpha/n$  (a~special case of (A1); sharp for density of regular functions~\cite{comi})  and with  a bounded, measurable, separated from zero weight $\omega$.  In this case we prove that $\mathcal{H}_{\cJp}$-Hausdorff measure zero sets are removable, where
\begin{align*}
\cJp(B_R(x_0))&= R^{-\theta} \int_{B_R(x_0)} R^{p(\theta-1)} + a(x)R^{q(\theta-1)} \, \dx\\
%&\leq c R^{\theta(p-1)} \int_{B_R(x_0)} R^{-p} + a(x)R^{-q} \ \dx\\
&\leq c \int_{B_R(x_0)} R^{-p\left(1-\frac{\theta}{q}(p-1)\right)}+a(x)^{1-\frac{\theta}{q}(p-1)}R^{-q\left(1-\frac{\theta}{q}(p-1)\right)}\, \dx,
\end{align*}
which implies the result of~\cite{ChDF}.

As a new study we analyze the borderline case between the double phase space and the variable exponent one, cf.~\cite{bacomi-st}. Namely, consider solutions to
\[0=-\diver A(x,Du)=-\diver\left(\omega(x)(\snr{Du}^{p-2}\big(1+a(x)\log({\rm e}+\snr{Du})\big)Du\right)\] with $1<p<\infty$, log-H\"older continuous $a$ and  a bounded, measurable, separated from zero weight $\omega$. Since this growth condition always satisfies (A0), (aInc)$_p$ and (aDec)$_{p+\varepsilon}$ with arbitrarily small $\varepsilon>0$ and it satisfies (A1) if the weight is H\"older continuous \cite[Proposition 7.2.5]{hahab}, our main result covers also this growth as a new result. We  provide removability of~$\mathcal{H}_{\cJp}$-Hausdorff measure zero sets with  
\begin{align*}
\cJp(B_R(x_0))= R^{-p+\theta(p-1)} \int_{B_R(x_0)} 1 + a(x)\log\left({\rm e}+ R^{\theta-1}\right) \, \dx.
\end{align*}

\subsubsection*{Orlicz growth operator.}
Having an $N$-function $B\in\Delta_2\cap\nabla_2$, we can allow for problems with the leading part of the operator with growth driven by $\vp(x,s)=B(s)$ with an example of \[0=-\diver \, A(x,Du)=-\diver\left(\omega (x)\tfrac{B(\snr{Du})}{\snr{Du}^2}Du\right)\]
with a bounded, measurable, and separated from zero weight $\omega$. Such growth conditions are equivalent to existence of the indices   $p$ and $q$ such that
\begin{align*}
1<p \leq \dfrac{B'(s)s}{B(s)}\leq q<\infty.
\end{align*}
Theorem~\ref{T6} provides removability of~$\mathcal{H}_{\cJp}$-Hausdorff measure zero sets with 
\begin{align*}
{\cJp}(B_R(x_0))\leq c R^{n-\theta}B(R^{\theta-1})
\end{align*}
and thus we improve the results from~\cite{chaly},   where the final claim follows from ours by the rough estimates expressed by the means of indices $p$ and $q$.

\subsubsection*{Other Musielak--Orlicz growth operators.}$ $ \\
To give more new examples one can consider problems stated in weighted Orlicz (if $\vp(x,s)=a(x)B(s)$), variable exponent double phase (if $\vp(x,s)=s^{p(x)}+a(x)s^{q(x)}$), or multi phase Orlicz cases (if $\vp(x,s)=\sum_i a_i(x)B_i(s)$), as long as $\vp(x,s)$ is comparable to a~function doubling  with respect to the second variable and it satisfies the non-degeneracy and continuity assumptions (A0)-(A1). Then, the size of removable sets is characterized in the general form provided in Theorem~\ref{T6} with the use of $\mathcal{H}_{\cJp}$-Hausdorff measure with ${\cJp}$ given by~\eqref{Jphi}.

\subsection*{Methods and remarks on the obstacle problem} The main steps of the proof follow the ideas of~\cite{kizo} adjusted to the inhomogeneous and general growth setting.  We shall use basic regularity properties of~solutions to the obstacle problem associated to~\eqref{A0}.  We consider the set
\begin{flalign}\label{con}
\mathcal{K}_{\psi,g}(\Omega):=\left\{ v\in W^{1,\vp(\cdot)}(\Omega)\colon v\ge \psi \ \  \mbox{a.e. in} \ \Omega \ \ \mbox{and} \ \ v-{g}\in W^{1,\vp(\cdot)}_{0}(\Omega)  \right\},
\end{flalign}
where $\psi\in W^{1,\vp(\cdot)}(\Omega)$ is the obstacle and $g \in W^{1,\vp(\cdot)}(\Omega)$ is the boundary datum. By a~solution to the obstacle problem we mean a function $v\in \mathcal{K}_{\psi,g}(\Omega)$ satisfying
\begin{flalign}\label{obs}
\int_{\Omega}A(x,Dv)\cdot D(w-v) \ \dx \ge 0 \ \ \mbox{for all } \ w\in \mathcal{K}_{\psi,g}(\Omega).
\end{flalign} 
By a supersolutions to \eqref{A0} we mean $\tilde{v}\in W^{1,\vp(\cdot)}(\Omega)$ satisfying
\begin{flalign}\label{sux}
\int_{\Omega}A(x,D\tilde{v})\cdot Dw \ \dx \ge0 \quad \mbox{for all non-negative} \ w\in W^{1,\vp(\cdot)}_{0}(\Omega).
\end{flalign}
Notice that a solution to problem \eqref{obs} is a supersolution to \eqref{A0}, we just need to test \eqref{obs} against $w:=v+\tilde{w}$, where $\tilde{w} \in W^{1,\vp(\cdot)}_{0}(\Omega)$ is any non-negative function. Since $\tilde{w}\in \mathcal{K}_{\psi,g}(\Omega)$,  the outcome is precisely the variational inequality \eqref{sux}. We note the following basic information on the existence and regularity for the obstacle problem, which are instrumental for us in the proof of Theorem~\ref{T6}.  Section~\ref{sec:obstacle} is devoted to the following results on solutions to the obstacle problem.

\begin{theorem}\label{T4} 
Suppose $A$ satisfies~\eqref{A-Car}--\eqref{A-monotone} in $\Omega\subset\rn$, $n\ge 2$, with a convex $\Phi$-function $\vp:\Omega\times\rp\to\rp$ satisfying (A0), (A1), (aInc)$_p$ and (aDec)$_q$ with some $1<p\leq q<\infty$. Let $\psi,g\in W^{1,\vp(\cdot)}(\Omega)$ be such that $\mathcal{K}_{\psi,g}(\Omega)\not =\emptyset$. Then, there exists a unique $v\in \mathcal{K}_{\psi,g}(\Omega)$, solution to the obstacle problem \eqref{obs}. Moreover, the following assertions hold true.
\begin{itemize} 
\item[-]{(Continuity and $\mathcal{A}_{\vp(\cdot)}$-harmonicity). If $\psi\in W^{1,\vp(\cdot)}(\Omega)\cap C(\Omega)$, then $v$ is continuous and solves \eqref{A0} in the open set $\{x\in \Omega\colon v(x)>\psi(x)\}$.}
\vspace{1mm}
\item[-]{(H\"older regularity). If $\psi \in W^{1,\vp(\cdot)}(\Omega)\cap C^{0,\theta}(\Omega)$ for some $\theta\in (0,1]$, then $v\in C^{0,\theta}_{\mathrm{loc}}(\Omega)$ and, for all open sets $\widetilde{\Omega}\Subset \Omega$, there holds
\begin{flalign}\label{v-Holder}
[v]_{0,\theta;\tilde{\Omega}}\le c(\data,\nr{\vp(\cdot,Dv)}_{L^{1}(\Omega)},\nr{\psi}_{L^{\infty}(\Omega)},[\psi]_{0,\theta})
\end{flalign} with $\data:=(n,c_1,c_2,p,q,L_p,L_q)$ -- the parameters describing the growth of $A$ and~$\vp$.}
\end{itemize}
\end{theorem}

\subsection*{Organization}  Section~\ref{sec:prelim} introduces main assumptions and the functional setting. In Section~\ref{sec:haus} we present the concept of intrinsic capacities and intrinsic Hausdorff--type measures. Section~\ref{sec:obstacle}  is devoted to the study on the obstacle problem, while Section~\ref{sec:rem} to the proof of the main result on the removability.

\vspace{2mm}
\section{Preliminaries}\label{sec:prelim}

\subsection{Notation}\label{sec:not}  We collect here basic remarks on the notation we use throughout the paper. Following a usual custom, we denote by $c$ a general constant larger than one. Different occurrences from line to line will be still denoted by $c$ {or similarly in special occurrences.} Relevant dependencies on parameters will be emphasized with parentheses, e.g. $c=c(n,p,q)$ means that $c$ depends on $n,p,q$. If $s>1$, by $s'$ we mean its H\"older conjugate, i.e. $s'=s/(s-1)$.
%, whereas by $t^{*}$ its Sobolev conjugate, i.e. $t^{*}=tn/(n-t)$. 
We denote by $B_{\rr}(x_{0}):=\left\{x\in \mathbb{R}^{n}\colon \snr{x-x_{0}}<\rr\right\}$ the open ball with center $x_{0}$ and radius $\rr>0$. When it is not important, or clear from the context, we shall omit denoting the center as follows: $B_{\rr}(x_{0})\equiv B_{\rr}$. Very often, when it is not otherwise stated, different balls will share the same center. Also, if $B$ is a ball with radius $\rr$, then $tB$ is a concentric ball with radius $t\rr$. When the ball $B$ is given we occasionally denote its radius as $\rr(B)$. With $U\subset \mathbb{R}^{n}$ being a~measurable set with finite and positive $n$-dimensional Lebesgue measure $\snr{U}>0$, and with $f\colon U\to \mathbb{R}^{k}$, $k\ge 1$ being a measurable map, by 
\begin{flalign*}
(f)_{U}:=\mint_{U}f(x) \ \dx =\frac{1}{\snr{U}}\int_{U}f(x) \ \dx
\end{flalign*}
we mean the integral average of $f$ over $U$. With $h\colon \Omega\to \mathbb{R}$, $U\subset \Omega$, and $\gamma \in (0,1]$ being a~given number, we shall denote
\begin{flalign*}
[h]_{0,\gamma;U}:=\sup_{\substack{x,y\in U,\\x\not =y}}\frac{\snr{h(x)-h(y)}}{\snr{x-y}^{\gamma}}, \qquad [h]_{0,\gamma}\equiv [h]_{0,\gamma;\Omega}.
\end{flalign*} 
{Recall that in~\eqref{con} we defined $\mathcal{K}_{\psi,g}(\Omega)$ with an obstacle $\psi$ and the boundary datum $g \in W^{1,\vp(\cdot)}(\Omega)$.  By $\mathcal{K}_{\psi}(\Omega)$ we denote $\mathcal{K}_{\psi,g}(\Omega)$ with $\psi\equiv g$.} 
%The notation $f\lesssim g$ means that there exists a constant $C>0$ such that $f\le C g$. The notation $f\approx g$ means that $f\lesssim g\lesssim f$. 
A function $f$ is \textit{almost increasing} if there exists a constant $L \ge 1$ such that $f(s) \le L f(t)$ for all $s \le t$ (more precisely, $L$-almost increasing). \textit{Almost decreasing} is defined analogously. 
%Two functions $\vp$ and $\psi$ are \textit{equivalent}, $\vp\simeq\psi$, if there exists $L\ge 1$ such that $\psi(x,\frac tL)\le \vp(x, t)\le \psi(x, Lt)$ for every $x \in \Omega$ and every $t>0$.

\subsection{Generalized Orlicz functions}\label{ssec:gen-or}
In order to capture within the same framework power, variable exponent, Orlicz, double phase growth and also more, we shall need to introduce the following bunch of definitions and remarks.
\begin{definition}
We say that $\vp: \Omega\times [0, \infty) \to [0, \infty]$ is a \textit{convex $\Phi$-function} ($\vp \in \Phi_c(\Omega)$), if
\begin{itemize}
\item for every $s \in [0, \infty)$ the function $x \mapsto \vp(x, s)$ is measurable;

\item  for almost every $x \in \Omega$ the function $s \mapsto \vp(x, s)$ is increasing, convex and left-continuous for $s>0$;

\item $\displaystyle \vp(x, 0) = \lim_{s \to 0^+} \vp(x,s) =0$ and $\displaystyle \lim_{s \to \infty}\vp(x,s)=\infty$ for almost every $x\in \Omega$;

\item The function $s \mapsto \frac{\vp(x, s)}s$ is $L$-almost increasing for $s>0$ for some $L\geq 1$ and almost every $x\in \Omega$.

\end{itemize}

\end{definition}

\noindent By $\vp^{-1}$ we denote the inverse of a convex $\Phi$-function $\vp$, that is
\[
\vp^{-1}(x,\tau) := \inf\{s \ge 0 \,:\, \vp(x,s)\ge \tau\}.
\] 
Let us write \[\vp^+_B (s) := \sup_{x \in B\cap \Omega} \vp(x, s)\qquad\text{ and
}\qquad \vp^-_B (s) := \inf_{x \in B \cap \Omega} \vp(x, s).\] 
Assume that the following two conditions hold.
\begin{itemize}
\item[(A0)] There exists $\beta \in(0,1)$ such that $\vp^+(\beta) \le 1 \le \vp^-(1/\beta)$.
\item[(A1)] There exists $\beta\in (0,1)$ such that
\[
\vp^+_B (\beta s) \le \vp^-_B (s)
\]
for every $s \in \big[1, (\vp_B^-)^{-1}(\tfrac1{|B|})\big]$ and every ball B with $\left (\vp^-_B\right)^{-1} \big (\frac{1}{|B|} \big ) \geq 1$.
\end{itemize}
Condition (A0) yields non-degeneracy, while (A1) restricts jumps $\vp$ can do.

We also introduce the following assumptions. %They are related to the $\Delta_2$ and $\nabla_2$ conditions from Orlicz space theory.  
\begin{itemize}
\item[$\ainc$] There exists $L_p\ge 1$ such that $s \mapsto \frac{\vp(x,s)}{s^{p}} $ is $L_p$-almost increasing in $(0,\infty)$.

\item[$\adec$] There exists $L_q\ge 1$ such that $s \mapsto \frac{\vp(x,s)}{s^{q}} $ is $L_q$-almost decreasing in $(0,\infty)$.
\end{itemize}
The function satisfying $\ainc$ and $\adec$ is called doubling. Let us motivate it. By the Fenchel--Young conjugate of $\vp$, we mean the function $\vp^{*}(x,t):=\sup_{s\ge 0}\left\{st-\vp(x,s)\right\}$.  We say that $\vp$ satisfies $\Delta_2$ condition (denoted by $\vp\in\Delta_2$), if there exists a constant $C\ge 1$ such that $\vp(x, 2s) \le C \vp(x, s)$ for every $x \in \Omega$ and every $s \geq 0$. If $\vp\in \Phi_c(\Omega)$, then (aDec)$_q$ is equivalent to $\vp\in\Delta_2$ \cite[Lemma 2.2.6]{hahab} and  (aInc)$_p$ is equivalent to $\vp^*\in\Delta_2$  \cite[Corollary~2.4.11]{hahab}. 

We note that if $\vp$ satisfies (A0), (A1) and (aDec)$_q$ with $q \leq n$, then $\vp$ satisfies so-called (A1-$n$) condition, that is
\begin{align*}
%\label{A1-n}
\vp^{+}_{B}(\beta s) \leq \vp^{-}_{B}(s) \quad \text{ for every } \quad s \in \left [1, \frac{1}{r(B)}\right ]
\end{align*} 
for every ball with $r(B) < 1$, see \cite[Lemma 2.9]{hahato}.  We can use doubling to transfer the small $\beta$ from the left-hand side to a (possibly) large constant $C=C(\beta, q, L_q)$ on the right-hand side, that is
\begin{align}
\label{A1-n}
\vp^{+}_{B}(s) \leq C \vp^{-}_{B}(s) \quad \text{ for every } \quad s \in \left [1, \frac{1}{r(B)}\right ].
\end{align}

 Direct consequences of the definition of $\vp^{*}$ are the following equivalence
\begin{flalign}\label{ex3}
\vp^{*}\left(x, {\vp(x,s)}/{s}\right)\leq \vp(x,s) \quad\text{for all }\ (x,s)\in \Omega \times \mathbb{R}^{n},
\end{flalign}
which for $\vp$ satisfying $\ainc$ and $\adec$ holds up to constants depending only on $p$ and $q$. We also note that $\vp$ satisfies (aInc)$_p$ or (aDec)$_q$ if and only if $\vp^{\ast}$ satisfies (aDec)$_{p'}$ or (aInc)$_{q'}$, respectively \cite[Proposition 2.4.9]{hahab}.

\subsection{Fuctional setting}\label{ssec:setting} There are various approaches how to describe generalized Orlicz spaces (called also Musielak-Orlicz spaces), cf.~\cite{hahab,IC-pocket,IC-b}. We aim at presenting below the main functional analytic tools in the least complicated way when the modular function $\vp\in \Phi_c(\Omega)$ satisfies assumptions (A0), (A1), (aInc)$_p$ and (aDec)$_q$ with some $1<p\leq q<\infty$. Let $L^0(\Omega)$ denote the set of measurable functions in $\Omega$. We define Musielak-Orlicz space 
\begin{flalign*}
L^{\vp(\cdot)}(\Omega):=\left\{w\in L^{0}(\Omega)\colon \int_{\Omega}\vp(x,|w|) \ \dx <\infty \right\},
\end{flalign*}
equipped with the Luxemburg norm
\begin{flalign*}
\nr{w}_{L^{\vp(\cdot)}(\Omega)}:=& 
\inf\left\{\lambda> 0\colon \int_{\Omega}\vp\left(x,\tfrac{1}{\lambda}|w|\right) \ \dx\le 1\right\}.
\end{flalign*}
 If $v\in L^{\vp(\cdot)}(\Omega)$ and $w\in L^{\vp^{*}(\cdot)}(\Omega)$, we have the H\"older inequality
\begin{flalign*}
\left | \ \int_{\Omega}vw \ \dx \  \right |\le&\ 2\,\nr{v}_{L^{\vp(\cdot)}(\Omega)}\nr{w}_{L^{\vp^{*}(\cdot)}(\Omega)}.
\end{flalign*}
We denote the modular \begin{flalign*}
\rho_{\vp(\cdot);\Omega}(w):= \int_{\Omega}\vp\left(x,w\right)\dx.
\end{flalign*}
If the doubling properties of $\vp$ are expressed by $\ainc$, $\adec$ and $a := \max\{L_p, L_q\}$, then
\begin{flalign}\nonumber
\min\left\{\left(\frac{1}{a}\rho_{\vp(\cdot);\Omega}(w)\right)^{\frac{1}{p}},\left(\frac{1}{a}\rho_{\vp(\cdot);\Omega}(w)\right)^{\frac{1}{q}}\right\}&\le \nr{w}_{L^{\vp(\cdot)}(\Omega)}\\
&\le\max \left\{\left(a\rho_{\vp(\cdot);\Omega}(w)\right)^{\frac{1}{p}},\left(a\rho_{\vp(\cdot);\Omega}(w)\right)^{\frac{1}{q}}\right\}\label{nr1}
\end{flalign}
and, {therefore} %equivalently,
\begin{flalign}\nonumber
{b_{1}}\min\left\{\left(\nr{w}_{L^{\vp(\cdot)}(\Omega)}\right)^{p},\left( \nr{w}_{L^{\vp(\cdot)}(\Omega)}\right)^{q}\right\}&\le \rho_{\vp(\cdot);\Omega}(w)\\
&%\qquad\qquad\qquad\qquad\qquad\qquad\qquad 
\le {b_{2}} \max \left\{\left(\nr{w}_{L^{\vp(\cdot)}(\Omega)}\right)^{p},\left(\nr{w}_{L^{\vp(\cdot)}(\Omega)}\right)^{q}\right\}\label{nr3}
\end{flalign}
{for some constants $b_1$ and $b_2$ {depending also only on the parameters $p,q,L_p,L_q$ describing the growth of $\vp(\cdot)$}.} 
Since the nonlinear tensor $A$ satisfies \eqref{A}, problem \eqref{A0} is naturally posed in the Musielak--Orlicz--Sobolev space
\begin{flalign*}
W^{1,\vp(\cdot)}(\Omega):=\left\{w\in W^{1,1}(\Omega)\colon 
\ w,\,|Dw|\in {L^{\vp(\cdot)}(\Omega)}\right\},
\end{flalign*}
equipped with the norm $ \nr{w}_{W^{1,\vp(\cdot)}(\Omega)}:=  \nr{w}_{L^{\vp(\cdot)}(\Omega)}+\nr{Dw}_{L^{\vp(\cdot)}(\Omega )}$. Upon such a definition $W^{1,\vp(\cdot)}(\Omega)$ is a Banach space, which, due to the doubling properties of $\vp(\cdot)$, is separable and reflexive. The dual space can be characterized by the means of the Fenchel-Young conjugate of $\vp(\cdot)$, namely we have $(W^{1,\vp(\cdot)}(\Omega))^{*}\sim W^{1,\vp^{*}(\cdot)}(\Omega)$.  Space $W^{1,\vp(\cdot)}_{\mathrm{loc}}(\Omega)$ is defined in the standard way.  We shall also define zero--trace space $W^{1,\vp(\cdot)}_{0}(\Omega)$ as a closure of $C^{\infty}_{c}(\Omega)$ functions in $W^{1,\vp(\cdot)}(\Omega)$. Justification of this choice of definition requires some comments, since it is known that in inhomogeneous spaces smooth functions may be not dense~\cite{eslemi,fomami,zh,badi}. 

\begin{remark}[Density]\label{rem:density}\rm In general, to get density of regular functions (smooth/Lipschitz) in norm in Musielak-Orlicz-Sobolev spaces, besides the (doubling) type of growth of $\vp(\cdot)$, what has to be controlled is its modulus of continuity (speed of growth has to be balanced with the regularity in the spacial variable), see~\cite{hahab} and~\cite{yags}. The critical role to get it here is played by assumption (A1) and, in turn, the definition of $W^{1,\vp(\cdot)}_{0}(\Omega)$ makes sense. In fact, the natural topology for Musielak--Orlicz--Sobolev spaces is the modular one, i.e. the one coming from  the notion of modular convergence~\cite{yags,IC-b,hahab}. We say that a sequence $(w_j)_{j \in \N}\subset L^{\vp(\cdot)}(\Omega)$ converges to $w$ modularly in $L^{\vp(\cdot)}(\Omega)$ if
\begin{flalign*}
\lim_{j\to \infty}w_{j}(x)=w(x) \ \ \mbox{for a.e.} \ x\in \Omega \quad \mbox{and}\quad \lim_{j\to \infty}\int_{\Omega}\vp(x,|w_{j}-w|) \, \dx=0.
\end{flalign*}
Consequently, $w_j\to w$ modularly in $W^{1,\vp(\cdot)}(\Omega)$ if both $w_j\to w$ and $Dw_j\to Dw$ modularly in $L^{\vp(\cdot)}(\Omega)$. Since the growth of $\vp$ is comparable to doubling, the modular convergence is equivalent to the norm convergence~\cite{IC-b,hahab}.  \end{remark}

\section{Intrinsic capacities and intrinsic Hausdorff measures}\label{sec:haus} 

\subsection{Definitions}\label{ssec:haus-def}
We define the \emph{intrinsic $\vp(\cdot)$-capacity} and recall its main features exactly in the form we need. Our main reference for this section is~\cite{bahaha}. Throughout this section we always assume that $\vp\in\Phi_c(\Omega)$ satisfies (A0), (A1), (aInc)$_p$ and (aDec)$_q$ with $1<p\leq q\leq n$.

 Given a~compact set $K\subset \Omega$, we denote its relative $\vp(\cdot)$-capacity as
\begin{flalign*}
cap_{\vp(\cdot)}(K,\Omega):=\inf_{f\in \mathcal{R}(K)}\int_{\Omega}\vp(x,|Df|) \ \dx,
\end{flalign*}
where the set of test functions is
\begin{flalign*}
\mathcal{R}_{\vp(\cdot)}(K):=\left\{ f\in W^{1,\vp(\cdot)}(\Omega)\cap C_{0}(\Omega)\colon \ \  f\ge 1 \ \ \mbox{in} \ \ K \right\}.
\end{flalign*}
As usual, for open subsets $U\subset \Omega$ and general $E\subset \Omega$ we have
\begin{flalign*}
cap_{\vp(\cdot)}(U,\Omega):=\sup_{{\substack{K\subset U,\\ K \ \mbox{compact}}}}cap_{\vp(\cdot)}(K,\Omega)
\end{flalign*}
and then
\begin{flalign*}
cap_{\vp(\cdot)}(E,\Omega):=\inf_{{\substack{E\subset U\subset \Omega, \\ U \ \mbox{open}}}}cap_{\vp(\cdot)}(U,\Omega).
\end{flalign*}
The structure of $\vp(\cdot)$ guarantees that $cap_{\vp(\cdot)}$ enjoys the standard properties of~Sobolev capacities. In particular,  as shown in \cite{bahaha} due to the convexity of $s\mapsto \vp(\cdot,s)$,  $cap_{\vp(\cdot)}$ is Choquet, which means that
\begin{flalign}\label{cho}
cap_{\vp(\cdot)}(E,\Omega)=\sup\left\{cap_{\vp(\cdot)}(K,\Omega)\colon K\subset E \ \mbox{is a compact set}\right\}.
\end{flalign}
Moreover, as $\vp$ satisfies (A0) and (A1), we see that the relative capacity $cap_{\vp(\cdot)}$ is equivalent to the capacity $C_{\vp(\cdot)}$ defined in \cite[Section~3]{bahaha}, see \cite[Theorem~7.3 and Proposition~7.5]{bahaha}.
\begin{remark}\label{r3}\rm Whenever we consider function $f\in \mathcal{R}_{\vp(\cdot)}(K)$ on  a compact set $K\Subset \Omega$ there is no loss of generality in assuming $0\le f\le 1$ on $\Omega$. {Since $f\in C_{0}(\Omega)$, $f\ge 1$ on $K$ and the map $t\mapsto \min\{t,1\}$ is Lipschitz,} it follows that $\tilde{f}:=\min\{f,1\} \in \mathcal{R}_{\vp(\cdot)}(K)$. Moreover,
\begin{flalign*}
\int_{\Omega}\vp(x,|D\tilde{f}|) \ \dx =&\int_{\{x\in \Omega\colon f(x)< 1\}}\vp(x,|Df|) \ \dx \le \int_{\Omega}\vp (x,|Df|) \ \dx,
\end{flalign*}
On the other hand, according to Remark~\ref{rem:density} yielding the density of smooth and compactly supported functions in $W^{1,\vp(\cdot)}_{0}(\Omega)$, there is also no loss of generality in restricting ourselves to $f\in C^{\infty}_{c}(\Omega)$.
\end{remark}

Naturally associated to these capacities is the concept of \emph{intrinsic Hausdorff measures}, introduced in \cite{demi}, see also \cite{ni,tu}. For any $n$-dimensional open ball $B\subset \Omega$  of radius $\rr(B)\in (0,\infty)$, we define
\begin{flalign*}
h_{\vp(\cdot)}(B):=\int_{B}\vp\left(x,\tfrac{1}{\rr(B)}\right) \ \dx.
\end{flalign*}Note that there is no difference in the following in taking closed balls here. Moreover, since $\vp$ satisfies (aInc)$_p$ and (aDec)$_q$ with $1<p\leq q\leq n$ we may apply the standard Carath\'eodory's construction to obtain an outer measure for any $E\subset \Omega$. We define the $\delta$-approximating Hausdorff measure of $E$, $\mathcal{H}_{\vp(\cdot),\delta}(E)$ with $\delta \leq 1$, by
\begin{flalign*}
\mathcal{H}_{\vp(\cdot),\delta}(E)=\inf_{\mathcal{C}_{E}^{\delta}}\sum_{j}h_{\vp(\cdot)}(B_{j}),
\end{flalign*}
where
\begin{flalign}
\label{C_E}
\mathcal{C}_{E}^{\delta}=\left \{\ \{B_{j}\}_{j\in \N}\ \mathrm{is \ a \ countable \ collection \ of \ balls \ B_j \subset \Omega \ covering \ }E,\, \ \rr(B_{j})\le \delta \  \right\}.
\end{flalign}
As $0<\delta_{1}<\delta_{2}<\infty$ implies $\mathcal{C}_{E}^{\delta_{1}}\subset \mathcal{C}_{E}^{\delta_{2}}$, we have that $\mathcal{H}_{\vp(\cdot), \delta_1}(E)\ge \mathcal{H}_{\vp(\cdot),\delta_{2}}(E)$ and there exists the limit
\begin{flalign*}
\mathcal{H}_{\vp(\cdot)}(E):=\lim_{\delta \to 0}\mathcal{H}_{\vp(\cdot),\delta}(E)=\sup_{\delta>0}\mathcal{H}_{\vp(\cdot),\delta}(E)\;.
\end{flalign*}
By standard arguments, found for example in \cite[2.10.1, p. 169]{fe}, $\mathcal{H}_{\vp(\cdot)}$ is a Borel regular measure. 
\begin{proposition}\cite[Theorem 2]{demi}\label{eq}
For $\vp\in\Phi_c(\Omega)$ under assumptions (A0), (A1), (aInc)$_p$ and (aDec)$_q$ for $1<p\leq q \leq n$, if $E\subset \mathbb{R}^{n}$ is such that $\mathcal{H}_{\vp(\cdot)}(E)<\infty$, then $cap_{\vp(\cdot)}(E)=0$.
\end{proposition}
\subsection{Properties of $\mathcal{H}_{\cJp}$}
To formulate our results of removable sets, we introduce an intrinsic Hausdorff measure $\mathcal{H}_{\cJp}$. It involves a transform of the $\vp(\cdot)$ given by the means of \eqref{Jphi}. This Hausdorff measure of a set $E$ is defined in the standard way
\begin{align*}
\mathcal{H}_{\cJp}(E) = \lim_{\delta \to 0} \inf_{C_E^{\delta}} \sum_{j} \rr_j^{-\theta} \int_{B_{\rr_j}} \vp(x, \rr_j^{\theta-1})\ \dx,
\end{align*}
where $C_E^{\delta}$ is defined in \eqref{C_E}. Since $\vp$ is doubling, we see that $\cJp$ is finite for any Euclidean ball and therefore \cite[2.10.1, p. 169]{fe} guarantees that $\mathcal{H}_{\cJp}$ generates a Borel regular measure.% enjoys the same structure as $\vp(\cdot)$, the resulted intrinsic Hausdorff measure $
%\mathcal{H}_{\cJp}$ has the same basic properties as 
%$\mathcal{H}_{\vp(\cdot)}$.

\subsection{Removability of sets of finite $\mathcal{H}_{\vp(\cdot)}$--measure}
The aim of this subsection is to show that the sets of finite $\mathcal{H}_{\vp(\cdot)}$--measure are removable for $\mathcal{A}_{\vp(\cdot)}$-harmonic functions. Following \cite[Chapter 2]{hekima} and Remark~\ref{rem:density}, for $E\subset \Omega$ relatively closed, we say that
\begin{flalign*}
W^{1,\vp(\cdot)}_{0}(\Omega)=W^{1,\vp(\cdot)}_{0}(\Omega\setminus E)
\end{flalign*}
if for any given $w\in W^{1,\vp(\cdot)}_{0}(\Omega)$ there exists a sequence $(w_{j})_{j\in \N}\subset C^{\infty}_{c}(\Omega\setminus E)$ such that  $w_j\to w$ modularly in $W^{1,\vp(\cdot)}_{0}(\Omega)$. Now we are ready to state our first two results, which clarify when a set is negligible in $W^{1,\vp(\cdot)}(\Omega)$. The already classical version of this fact stated in the Sobolev space $W^{1,p}$ can be found in \cite[Section 2.42]{hekima}. We recall here that since the growth of $\vp$ is doubling, Remark~\ref{rem:density} explains that we can work with the modular convergence.
\begin{lemma}\label{l2}
Suppose that $E$ is a relatively closed subset of $\Omega$. Then 
\begin{flalign*}
W^{1,\vp(\cdot)}_{0}(\Omega)=W^{1,\vp(\cdot)}_0(\Omega\setminus E)\ \ \mbox{if and only if} \ \ cap_{\vp(\cdot)}(E,\Omega)=0.
\end{flalign*}
\end{lemma}
\begin{proof} Assume first that $cap_{\vp(\cdot)}(E,\Omega)=0$. Obviously, $W^{1,\vp(\cdot)}_0(\Omega\setminus E)\subset W^{1,\vp(\cdot)}_0(\Omega)$, so it suffices to show that $W^{1,\vp(\cdot)}_0(\Omega)\subset W^{1,\vp(\cdot)}_0(\Omega\setminus E)$. Since $cap_{\vp(\cdot)}(E,\Omega)=0$, according to Remark \ref{r3}, there exists a sequence $(f_{j})_{j \in \N}\subset \left(\mathcal{R}_{\vp(\cdot)}(E)\cap C^{\infty}_{c}(\Omega)\right)$ such that
\begin{flalign}\label{fj}
0\le f_{j}\le 1 \ \ \mbox{and} \ \ \lim_{j\to \infty}\int_{\Omega}\vp(x,|Df_{j}|) \ \dx = 0.
\end{flalign}
Moreover, having $\phi \in C^{\infty}_{c}(\Omega)$ for any $j\in \N$, the map $\psi_{j}:=(1-f_{j})\phi$ has support contained in $\Omega\setminus E$. Then we have $(\psi_{j})_{j \in \N}\subset C^{\infty}_{c}(\Omega\setminus E)$. The dominated convergence theorem implies that
\begin{flalign*}
\int_{\Omega\setminus E}\vp(x,|D\psi_{j}-D\phi|) \ \dx=0,
\end{flalign*}
therefore $\vp\in W^{1,\vp(\cdot)}_{0}(\Omega\setminus E)$. Since by Remark~\ref{rem:density} and the dominated convergence theorem we can approximate any $w \in W^{1,\vp(\cdot)}_{0}(\Omega \setminus E)$ in the modular {topology} %convergence
via the sequence of truncations $({w}_{k})_{k\in \mathbb{N}}:=(\max \{-k,\min\{w,k\} \})_{k \in \mathbb{N}}$ we have
\begin{flalign*}
W^{1,\vp(\cdot)}_{0}(\Omega)\subset W^{1,\vp(\cdot)}_{0}(\Omega\setminus E),
\end{flalign*}
and the `if' part of the lemma is proven. 

For the `only if' part, by the Choquet property~\eqref{cho}, it is sufficient to show that any compact $K\subset E$ we have $cap_{\vp(\cdot)}(K,\Omega)=0$. Let us fix an arbitrary $f \in \mathcal{R}_{\vp(\cdot)}(K)$. Since $W^{1,\vp(\cdot)}_{0}(\Omega)=W^{1,\vp(\cdot)}_{0}(\Omega\setminus E)$, there exists a sequence $(\phi_{j})_{j \in \N}\subset  C^{\infty}_{c}(\Omega\setminus E)$ such that {$\phi_j\to f$ a.e. in $\Omega$ and} $\lim_{j\to \infty}\int_{\Omega}\vp(x,|D\phi_{j}-Df|) \, \dx=0.$ Therefore, $g_{j}:=f-\phi_{j}\in \mathcal{R}_{\vp(\cdot)}(K,\Omega)$ for all $j\in \N$. As a consequence of the definition of the capacity $cap_{\vp(\cdot)}$, we have
\begin{equation*}
cap_{\vp(\cdot)}(K,\Omega)\le \lim_{j\to \infty}\int_{\Omega}\vp(x,|Dg_{j}|) \ \dx=0. \qedhere
\end{equation*}
\end{proof}
As a direct consequence of Lemma \ref{l2}, we show that sets of finite $\mathcal{H}_{\vp(\cdot)}$--measure are removable for $\mathcal{A}_{\vp(\cdot)}$-harmonic maps.
\begin{corollary}\label{coro:H0}
Let $E\subset \Omega$ be a relatively closed subset of $\Omega$ such that {$\mathcal{H}_{\vp(\cdot)}(E)<\infty$} and $u\in W^{1,\vp(\cdot)}(\Omega)$ satisfying
\begin{flalign}\label{70}
\int_{\Omega\setminus E}A(x,Du)\cdot Dw \ \dx=0 
\end{flalign}
for all $w\in W^{1,\vp(\cdot)}_{0}(\Omega\setminus E)$. Then, $u$ is a solution to \eqref{A0} on the whole $\Omega$.
\end{corollary}
\begin{proof}
Since $\mathcal{H}_{\vp(\cdot)}(E)<\infty$, by Proposition \ref{eq} we have that $cap_{\vp(\cdot)}(E,\Omega)=0$, thus by Lemma~\ref{l2} we can conclude that $W^{1,\vp(\cdot)}_{0}(\Omega\setminus E)=W^{1,\vp(\cdot)}_{0}(\Omega)$, so \eqref{70} actually holds for all $w \in W^{1,\vp(\cdot)}_{0}(\Omega)$.
\end{proof}

\section{The obstacle problem}
\label{sec:obstacle} The well-posedness of the obstacle problems results from considerations based on general functional analytic approach from~\cite{kiod}. Continuity and $\mathcal{A}_{\vp(\cdot)}$-harmonicity of its solutions outside of the contact set is commented further below as a consequence of the reasoning provided for \cite[Theorem 5.8]{ka}, while H\"older regularity is taken from \cite[Theorem 5.3]{kale}. 
\subsection{Well-posedness of the obstacle problem}\label{sec:obst}
In this section we comment on the existence and uniqueness to the obstacle problem related to~\eqref{A0} with the structure~\eqref{A-Car}--\eqref{A-monotone}. They can be easily obtained   as a consequence of classical results on solvability in reflexive Banach spaces and comparison principles, because of the properties of the operator $\mathcal{A}_{\vp(\cdot)}$ defined in~\eqref{calAH}. We assume $\mathcal{K}_{\psi,g}(\Omega)\not =\emptyset$, see~\eqref{con}. Notice that when $\psi\equiv g$, we have $\mathcal{K}_{\psi,g}(\Omega)\not =\emptyset$ since $\psi\in \mathcal{K}_{\psi,g}(\Omega)$. Our result reads as follows.

\begin{proposition}\label{prop:ex-uni} Let $\vp\in\Phi_c(\Omega)$ be a function such that (A0), (A1), (aInc)$_p$ and (aDec)$_q$ hold true with some $1<p\leq q<\infty$. Suppose that $A$ satisfying \eqref{A-Car}--\eqref{A-monotone}  and $\psi,g\in W^{1,\vp(\cdot)}(\Omega)$ are such that $\mathcal{K}_{\psi,g}(\Omega)\not =\emptyset$. Then there exists a~unique weak solution $v\in\mathcal{K}_{\psi,g}(\Omega)$ to problem~\eqref{obs}.
\end{proposition}

We recall some elementary facts about monotone operators defined on a reflexive Banach space, which finally will be applied to the operator $\mathcal{A}_{\vp(\cdot)}$, defined in~\eqref{calAH}.
  
\begin{definition}\label{d1}
Let $X$ be a reflexive Banach space with dual $X^{*}$ and $\langle \cdot,\cdot\rangle$ denote a pairing between $X^{*}$ and $X$. If $K\subset X$ is any closed, convex subset, then a map $T\colon K\to X^{*}$ is called monotone if it satisfies $
\langle Tw-Tv,w-v \rangle\ge 0 \ \ \mbox{for all} \ \ w,v\in K.
$ Moreover, we say that $T$ is coercive if there exists a $w_{0}\in K$ such that
\begin{flalign*}
\lim_{\nr{w}_{X}\to \infty}\frac{\langle Tw,w-w_{0} \rangle}{\nr{w}_{X}}=\infty \ \ \mbox{for all} \ \ w\in K.
\end{flalign*}
\end{definition}
The following proposition guarantees the existence of solution to variational inequalities associated to monotone operators.
\begin{proposition}\cite{kiod} \label{exist}
Let $K\subset X$ be a nonempty, closed, convex subset in a separable and reflexive space $X$. Assume further that $T\colon K\to X^{*}$ is monotone, weakly continuous and coercive on $K$. Then there exists an element $v\in K$ such that $\langle Tv,w-v\rangle\ge 0$ for all $w\in K$. 
\end{proposition}

Let us prepare ourselves to apply the above result. For all of them we assume assumptions of Theorem~\ref{T4}.

\begin{remark}[The operator]\rm \label{rem:op} We notice that $\mathcal{A}_{\vp(\cdot)}$ is defined on a reflexive and separable Banach space $W^{1,\vp(\cdot)}(\Omega)$. As a matter of fact, $\mathcal{A}_{\vp(\cdot)}(W^{1,\vp(\cdot)}(\Omega))\subset (W^{1,\vp(\cdot)}(\Omega))^{*}$. Indeed, when $v\in W^{1,\vp(\cdot)}(\Omega)$ and $w\in C_0^\infty(\Omega)$~\eqref{ex3} and Poincar\'e inequality~\cite[Theorem~6.2.8]{hahab} justify that
\begin{flalign}\label{i1}
\snr{\langle \mathcal{A}_{\vp(\cdot)}v,w \rangle}\le &c_1\int_{\Omega}\frac{\vp(x,\snr{Dv})}{\snr{Dv}}\snr{Dw} \ \dx \le c\left \| \frac{\vp(\cdot,\snr{Dv})}{\snr{Dv}}\right \|_{L^{\vp^{*}(\cdot)}(\Omega)}\nr{Dw}_{L^{\vp(\cdot)}(\Omega)}\nonumber \\
\le &c\nr{Dv}_{L^{\vp(\cdot)}(\Omega)}\nr{Dw}_{L^{\vp(\cdot)}(\Omega)}\le c\nr{w}_{W^{1,\vp(\cdot)}(\Omega)}.
\end{flalign}
\end{remark}

\begin{lemma}\label{lem:A-vp-weak-cont} Having $\vp$ and $ A$ as in Proposition~\ref{prop:ex-uni}, the operator $\mathcal{A}_{\vp(\cdot)}$ is weakly continuous on $W^{1,\vp(\cdot)}(\Omega)$.
\end{lemma}
\begin{proof} Let $v,(v_{j})_{j \in \N}\subset W^{1,\vp(\cdot)}(\Omega)$ be such that $v_{j}\to v \ \ \mbox{in} \ \ W^{1,\vp(\cdot)}(\Omega)$. Then it is bounded in $ W^{1,\vp(\cdot)}(\Omega)$ and, up to a subsequence, $\lim_{j\to \infty}v_{j}(x)=v(x)$ and $\lim_{j\to \infty}Dv_{j}(x)= Dv(x)$ for almost all $x\in \Omega$. Since $z\mapsto A(\cdot,z)$ is continuous, for any $w\in W^{1,\vp(\cdot)}(\Omega)$, we have convergence \[\lim_{j\to \infty}A(x,Dv_{j}(x))\cdot Dw(x)=A(x,Dv(x))\cdot Dw(x)\quad\text{ for a.e. }x\in \Omega\,.\] Moreover, if $E\subset \Omega$ is any measurable subset, then  {as in~\eqref{i1} we have}
\begin{flalign*}
\left | \  \int_{E}A(x,Dv_{j})\cdot Dw \ \dx \ \right |\le &\int_{E}\snr{A(x,Dv_{j})\cdot Dw} \ \dx \le c\nr{Dw}_{L^{\vp(\cdot)}(E)}\,.
\end{flalign*}
We can now apply Vitali's convergence theorem to conclude
\begin{equation*}
\lim_{j\to \infty} \langle \mathcal{A}_{\vp(\cdot)}v_{j},w \rangle=\lim_{j\to \infty}\int_{\Omega}A(x,Dv_{j})\cdot Dw \ \dx =\int_{\Omega}A(x,Dv)\cdot Dw \ \dx=\langle \mathcal{A}_{\vp(\cdot)}v,w \rangle\,. \qedhere
\end{equation*} 

\end{proof}
\begin{lemma}\label{lem:K-is-conv}
Having $\vp,  \psi,g$ as in Proposition~\ref{prop:ex-uni}, the set $\mathcal{K}_{\psi,g}(\Omega)\subset W^{1,\vp(\cdot)}(\Omega)$  is closed and convex.
\end{lemma}
\begin{proof} When for $\lambda \in [0,1]$ and $w_{1},w_{2}\in \mathcal{K}_{\psi,g}(\Omega)$ we define $w_{\lambda}:=\lambda w_{1}+(1-\lambda)w_{2}$, then $w_{\lambda}\in W^{1,\vp(\cdot)}(\Omega)$, $w_{\lambda}-g\in W^{1,\vp(\cdot)}_{0}(\Omega)$ and $w_{\lambda}\ge \psi$ a.e. in $\Omega$. Moreover, if $w\in W^{1,\vp(\cdot)}(\Omega)$ and $(w_{j})_{j \in \N}\subset \mathcal{K}_{\psi,g}(\Omega)$ is any sequence such that $\lim_{j\to \infty}\int_{\Omega}\vp(x,|Dw_{j}-Dw|) \ \dx=0$, then by the continuity of the trace operator, $w-g\in W^{1,\vp(\cdot)}_{0}(\Omega)$ and, by Lebesgue's dominated convergence theorem, $w\ge \psi$ a.e. in $\Omega$. 
\end{proof}

\begin{lemma}\label{lem:comp-princ} Having $\vp, A, \psi,g$ as in Proposition~\ref{prop:ex-uni}, if $v\in W^{1,\vp(\cdot)}(\Omega)$ is a solution to problem \eqref{obs}, $\tilde{v}\in W^{1,\vp(\cdot)}(\Omega)$ is a supersolution to \eqref{A0}, then $\tilde{v}(x)\ge v(x)$ for a.e. $x\in \Omega$.
 \end{lemma}
\begin{proof} We set $w:=\min\{v,\tilde{v}\}\in \mathcal{K}_{\psi,g}(\Omega)$ and note that the map $\tilde{w}:=v-\min\{\tilde{v},v\}$ is an admissible test in \eqref{sux}. As $v$ is a solution to \eqref{obs} and $w\in \mathcal{K}_{\psi,g}(\Omega)$ we have
\begin{flalign*}
\begin{cases}
\ \int_{\Omega\cap \{x\colon \tilde{v}(x)< v(x) \}}A(x,D\tilde{v})\cdot (Dv-D\tilde{v}) \ \dx \ge 0,\\
\ \int_{\Omega\cap \{x\colon \tilde{v}(x)< v(x)\}}A(x,Dv)\cdot (D\tilde{v}-Dv) \ \dx\ge 0.
\end{cases}
\end{flalign*}
Adding the two inequalities in the above display and {using \eqref{A-monotone}}, we obtain
\begin{flalign*}
0\le \int_{\Omega\cap \{x\colon \tilde{v}(x)< v(x)\}}\left(A(x,D\tilde{v})-A(x,Dv)\right)\cdot (Dv-D\tilde{v}) \ \dx \le 0,
\end{flalign*}
thus either $\snr{\Omega\cap \{x\colon \tilde{v}(x)< v(x)\}}=0$ or $D\tilde{v}=Dv$ a.e. on $\Omega\cap \{x\colon \tilde{v}(x)< v(x)\}$. This second alternative is excluded by the fact that $w\in \mathcal{K}_{\psi,g}(\Omega)$, so $v-\tilde{v}\in W^{1,\vp(\cdot)}_{0}(\Omega\cap \{x\colon \tilde{v}(x)< v(x)\})$. Therefore $\snr{\Omega\cap \{x\colon \tilde{v}(x)< v(x)\}}=0$ and $\tilde{v}\ge v$ a.e. in $\Omega$.
\end{proof}
We have the following direct consequence.
\begin{remark}\label{rem-i} If $u\in W^{1,\vp(\cdot)}(\Omega)$ is a solution to \eqref{A0}, it is % course
a supersolution to the same equation. Thus, whenever $v\in \mathcal{K}_{u}(\Omega)$ is a solution to problem \eqref{obs}, then $u(x)\ge v(x)$ for a.e. $x\in \Omega$.
\end{remark}

\begin{lemma}\label{lem:uniq}  Having $\vp, A, \psi,g$ as in Proposition~\ref{prop:ex-uni}, the solutions to~\eqref{obs} are unique.
\end{lemma}

\begin{proof} If there were two solutions $v_{1},v_{2}\in \mathcal{K}_{\psi,g}(\Omega)$ then each of them is an admissible test function for the other one. Using \eqref{A-monotone} we obtain
\begin{flalign*}
0\le \int_{\Omega}\left(A(x,Dv_{1})-A(x,Dv_{2})\right)\cdot (Dv_{2}-Dv_{1}) \ \dx \le 0.
\end{flalign*}
Hence, $Dv_{1}(x)=Dv_{2}(x)$ for a.e. $x\in \Omega$ and since $v_{1}-v_{2}\in W^{1,\vp(\cdot)}_{0}(\Omega)$, we can conclude that $v_{1}=v_{2}$ almost everywhere.
\end{proof}

\begin{lemma}\label{lem:coerc} Having $\vp, A, \psi,g$ as in Proposition~\ref{prop:ex-uni}, $\mathcal{A}_{\vp(\cdot)}$ is coercive over
\begin{flalign*}
\mathcal{K}^{\Lambda}_{\psi,g}(\Omega):=\mathcal{K}_{\psi,g}(\Omega)\cap \left\{w\in W^{1,\vp(\cdot)}(\Omega)\colon\ \ \nr{Dw}_{L^{\vp(\cdot)}(\Omega)}\le \Lambda\right\},\quad \text{for any }\  \Lambda \ge 0.
\end{flalign*} 
\end{lemma}
\begin{proof}
We fix $w,w_{0}\in \mathcal{K}_{\psi,g}^\Lambda(\Omega)$ and, using $\eqref{A}_{2}$ and H\"older's and Young's inequalities we obtain
\begin{flalign*}
\langle\mathcal{A}_{\vp(\cdot)}w,w-w_{0} \rangle=&\int_{\Omega}A(x,Dw)\cdot(Dw-Dw_{0}) \ \dx\nonumber \\
\ge &c_2\int_{\Omega}\vp(x,\snr{Dw}) \ \dx -2\nr{Dw}_{L^{\vp(\cdot)(\Omega)}}\nr{Dw_0}_{L^{\vp(\cdot)(\Omega)}}.
\end{flalign*}
From \eqref{nr3} we have
\begin{flalign*}
\int_{\Omega}\vp(x,Dw) \ \dx\ge {b_{1}}\min\left\{\nr{Dw}_{L^{\vp(\cdot)}(\Omega)}^{p},\nr{Dw}_{L^{\vp(\cdot)}(\Omega)}^{q}\right\},
\end{flalign*} 
therefore, merging the content of the two previous displays we obtain
\begin{flalign}\label{66}
\frac{\langle \mathcal{A}_{\vp(\cdot)}w,w-w_{0}\rangle}{\nr{Dw}_{L^{\vp(\cdot)}(\Omega)}}\ge& {c_2 b_{1}} \min\left\{\nr{Dw}_{L^{\vp(\cdot)}(\Omega)}^{p-1},\nr{Dw}_{L^{\vp(\cdot)}(\Omega)}^{q-1}\right\}-2{\nr{Dw_0}_{L^{\vp(\cdot)}(\Omega)}}\to\infty
\end{flalign}
as $\nr{Dw}_{L^{\vp(\cdot)}(\Omega)}\to \infty$.  
\end{proof}

Having the above, we are in the position to prove well--posedness of the obstacle problem.

\begin{proof}[Proof of Proposition~\ref{prop:ex-uni}]  
To get existence we apply Proposition~\ref{exist}. Let us verify its assumptions. We have an operator $\mathcal{A}_{\vp(\cdot)}$ defined on a reflexive Banach space $W^{1,\vp(\cdot)}(\Omega)$, which due to Lemma~\ref{lem:A-vp-weak-cont} is weakly continuous and because of~\eqref{A-monotone} is monotone. According to Remark~\ref{rem:op} $\mathcal{A}_{\vp(\cdot)}(W^{1,\vp(\cdot)}(\Omega))\subset (W^{1,\vp(\cdot)}(\Omega))^{*}$, Lemma~\ref{lem:K-is-conv} provides that $\mathcal{K}_{\psi,g}(\Omega)\subset W^{1,\vp(\cdot)}(\Omega)$  is closed and convex. On the other hand, \eqref{nr1} implies that 
\begin{flalign*}
\mathcal{K}^{\Lambda}_{\psi,g}(\Omega)\subset \left\{w\in W^{1,\vp(\cdot)}(\Omega)\colon \nr{w}_{W^{1,\vp(\cdot)}(\Omega)}\le \bar{c}=\bar{c}(n,p,q,\Lambda,\nr{g}_{W^{1,\vp(\cdot)}(\Omega)},\diam(\Omega))\right\},
\end{flalign*}
thus $\mathcal{K}_{\psi,g}^{\Lambda}(\Omega)$ is bounded (and, of course, closed and convex) in $W^{1,\vp(\cdot)}(\Omega)$. Therefore, Proposition \ref{exist} yields that there exists a solution to problem \eqref{obs}, which due to Lemma~\ref{lem:uniq} is unique.
 \end{proof}

\subsection{H\"older regularity of the obstacle problem}

In this subsection we provide proofs for regularity of the solution to an obstacle problem. We start this process by showing that our solution satisfies an intrinsic Caccioppoli inequality.

\begin{proposition}
Suppose $A$ satisfies \eqref{A-Car}--\eqref{A-monotone} with $\vp\in \Phi_c(\Omega)$ satisfying (aDec)$_q$ for some $1< q < \infty$. If $B_r \Subset B_R \subset \Omega$ and  $v$ is a solution to the obstacle problem  to \eqref{obs}
\begin{align*}
\int_{\{v \geq k\} \cap B_R} \vp(x, |D(v-k)_+ |) \ \dx \leq c \int_{\{v \geq k\} \cap B_R} \vp\left (x,\dfrac{(v-k)_+}{R-r}\right ) \ \dx
\end{align*}
where $k \geq \sup_{x \in B_R} \psi(x)$ and $A(k,r)=\{v \geq k\} \cap B_R$.
\end{proposition}

\begin{proof} As $v$ is a solution to the obstacle problem also locally (by choosing a proper test function involving a cut-off function) we can test the equation with any $w\in \mathcal{K}_{\psi,g}(B_R)$ and have
\begin{align*}
\int_{\Omega} A(x, Dv) \cdot Dv \ \dx \leq \int_{B_R}  A(x, Dv) \cdot D w \ \dx.
\end{align*}
Using \eqref{A} on the left-hand side and \eqref{ex3} on the right-hand side to get
\begin{align*}
\int_{B_R} \vp(x, |Dv|) \ \ \dx &\leq \dfrac{c_1}{c_2} \int_{B_R} \dfrac{\vp(x, |Dv|)}{|Dv|} |Dw| \ \dx\\
& \leq \dfrac{1}{2}\int_{B_R} \vp^{\ast}\left (x,\dfrac{\vp(x, |Dv|)}{|Dv|}\right ) \ \dx + c \int_{B_R}\vp (x, |Dw|) \ \dx \\
&\leq \dfrac{1}{2}\int_{B_R} \vp(x, |Dv|) \ \dx + c\int_{B_R} \vp(x, |Dw|) \ \dx.
\end{align*}
Absorbing the first term on the right-hand side to the left-hand side we have
\begin{align*}
\int_{B_R} \vp(x, |Dv|) \ \ \dx \leq c\int_{B_R} \vp(x, |Dw|) \ \dx.
\end{align*}
We choose $w:= u-\eta(u-k)_+$, where $\eta \in C_c^{\infty}(B_R)$ is a standard cut-off function with $\eta=1$ in $B_r$ and $|D\eta| \leq \frac{2}{R-r}$. We note that $w\geq \psi$ by the assumption on $k$. With these choices, the proof follows exactly the same lines as in \cite[Lemma 4.3]{hahato}.
\end{proof}

Next step of proving regularity is to show that for bounded obstacles the solution is also bounded. The proof can be found in \cite{kale}. 

\begin{proposition}[\cite{kale}]\label{prop:v-bounded}
Suppose $\vp \in \Phi_c(\Omega)$ satisfies (A0), (A1), (aInc)$_{p}$ and (aDec)$_{q}$ for some $1 < p \leq q <\infty$. Under assumptions \eqref{A-Car}--\eqref{A-monotone} let $\psi,g\in W^{1,\vp(\cdot)}(\Omega)$ be such that $\mathcal{K}_{\psi,g}(\Omega)\not =\emptyset$ and let $v$ be a solution to the obstacle problem~\eqref{obs} such that $\rho_{{\vp(\cdot)};B_{2\rr}}(|D v|) \le 1$. Then 
if $\psi \in W^{1,\vp(\cdot)}(\Omega)\cap L^{\infty}(\Omega)$ and {$\sigma \in [\frac{1}{2}, 1)$},
then $v\in L^{\infty}_{\mathrm{loc}}(\Omega)$ and% , for any open set $\tilde{\Omega}\Subset \Omega$,
\begin{align}\label{v-bounded}
\eup_{B_{\sigma \rr}} (v-\ell)_+ \leq C (1-\sigma)^{-4nq^2}\left[  \left(\mint_{B_{\rr}} (v-\ell)_+^{q} \, \dx \right )^{1/q} + \left|(v-\ell)_{B_{\rr/2}}\right| + \rr \right]
\end{align}
for  any $B_{2\rr}\subset \Omega$
and $\ell \geq \sup_{B_{2\rr}} \psi$. The term $|(v-\ell)_{B_{\rr/2}}|$ can be omitted if $(v(x)-\ell)\geq 0$ for every $x\in B_{\rr/2}$.
\end{proposition}
\noindent Note that above $v$ is the solution in $\Omega$ and thus  $\rho_{{\vp(\cdot)};\Omega}(|D v|) <\infty$. The modular condition is stated for some ball $B_{2\rr}$, which can always be satisfied as long as the radius is chosen small enough due to absolute continuity of the integral.

\medskip

As solution to an obstacle problem is always a superminimizer, following the same lines as in \cite{hahale} we have the following weak Harnack inequality.

\begin{proposition}
\label{prop:inf-estimate}
Let $\vp\in\Phi_c(\Omega)$ 
satisfy (A0), (A1), (aInc)$_p$ and (aDec)$_q$ for some $1<p \leq q < \infty$.
Under assumptions \eqref{A-Car}--\eqref{A-monotone} let $\psi,g\in W^{1,\vp(\cdot)}(\Omega)$ be such that $\mathcal{K}_{\psi,g}(\Omega)\not =\emptyset$ and let $v$ be a {non-negative} solution to the obstacle problem  to \eqref{obs} such that $\rho_{{\vp(\cdot)};B_{2\rr}}(|D v|) \le 1$.
Then there exists $h_0>0$ such that for every $B_{2\rr}\subset \Omega$ we have
\begin{align}\label{32}
 \left( \mint_{B_\rr} v^{h_0}\,dx\right)^\frac1{h_0}
 \le c \left( 
 \eif_{B_{\rr/2}} v +  \rr\right).
\end{align}
\end{proposition}

These results altogether imply the H\"older continuity of $v$ {using a similar reasoning as in the proof of Lemma \ref{osc}.}

\begin{proposition}[\cite{kale}]\label{prop:hold}
Suppose $\vp \in \Phi_c(\Omega)$ satisfies (A0), (A1), (aInc)$_p$ and (aDec)$_q$ for some $1<p \leq q < \infty$. Under assumptions \eqref{A-Car}--\eqref{A-monotone} let $\psi,g\in W^{1,\vp(\cdot)}(\Omega)$ be such that $\mathcal{K}_{\psi,g}(\Omega)\not =\emptyset$. If $\psi \in W^{1,\vp(\cdot)}(\Omega)\cap C^{0,\beta_{0}}(\Omega)$ for some $\beta_{0}\in (0,1]$ and $v$ is a solution to the obstacle problem  to \eqref{obs} such that $\rho_{{\vp(\cdot)};B_{2\rr}}(|D v|) \le 1$, then $v\in C^{0,\beta_{0}}_{\mathrm{loc}}(\Omega)$ and, for all open sets $\tilde{\Omega}\Subset \Omega$, there holds
\begin{align}\label{v-Holder}
[v]_{0,\beta_{0};\tilde{\Omega}}\le c(\data,\nr{\vp(\cdot,\snr{Dv})}_{L^{1}(\Omega)},\nr{\psi}_{L^{\infty}(\Omega)},[\psi]_{0,\beta_{0}}).
\end{align}
\end{proposition} 

\begin{proof}[Proof of Theorem~\ref{T4}] It suffices to recall Propositions~\ref{prop:ex-uni} and \ref{prop:hold} to get the final claim.
\end{proof}

\section{Removable sets}\label{sec:rem}
In this last section we prove our main result, i.e. Theorem \ref{T6}. 
\subsection{Auxiliary results}
To begin we show a Caccioppoli-type inequality for non-negative supersolutions to \eqref{A0}.
\begin{lemma}\label{lem-cacc}
Let $\vp \in \Phi_c$ satisfy (A0), (A1), (aInc)$_p$ and (aDec)$_q$ with some $1<p\leq q<\infty$. Under assumptions \eqref{A-Car}--\eqref{A-monotone}, let $B_{\rr}\Subset \Omega$ be any ball, $\tilde{v}\in W^{1,\vp(\cdot)}(\Omega)$ a supersolution to \eqref{A0}, non-negative in $B_{\rr}$ and $\eta\in C^{1}_{c}(B_{\rr})$. Then for all $\gamma\in(1,p)$ there holds
\begin{flalign*}
\int_{B_{\rr}}\tilde{v}^{-\gamma}\eta^{q}\vp(x,\snr{D\tilde{v}}) \ \dx\le c\int_{B_{\rr}}\tilde{v}^{-\gamma}\vp(x,\snr{D\eta}\tilde{v}) \ \dx,
\end{flalign*}
with $c=c(c_1,c_2,p,q,\gamma)$.
\end{lemma}
\begin{proof}
Since $\tilde{v}$ is a non-negative supersolution to \eqref{A0},  by the comparison principle, either $\tilde{v}\equiv0$ a.e. on $B_{2\rr}$, or we can assume that $\tilde{v}$ is strictly positive in $B_{\rr}$. In the first scenario there is nothing interesting to prove, so we can look at the second one. For $\eta$ as in the statement, and any $\tilde{\gamma}>0$, we test \eqref{sux} against $w:=\eta^{q}\tilde{v}^{-\tilde{\gamma}}$ to obtain, with the help of $\eqref{A}_{1,2}$ and Young's inequality,
\begin{flalign}\label{41}
c_2 \tilde{\gamma}&\int_{B_{\rr}}\tilde{v}^{-\tilde{\gamma}-1}\eta^{q}\vp(x,\snr{D\tilde{v}}) \ \dx \le c_1 q\int_{B_{\rr}}\left(\frac{\vp(x,\snr{D\tilde{v}})}{\snr{D\tilde{v}}}\eta^{q-1}\snr{D\eta}\tilde{v}\right)\tilde{v}^{-\tilde{\gamma}-1} \ \dx\nonumber \\
& \le \frac{c_2 \tilde{\gamma}}{2}\int_{B_{\rr}}\tilde{v}^{-\tilde{\gamma}-1}\vp(x,\snr{D\tilde{v}})\eta^{q} \ \dx +\left(\frac{c}{c_2\tilde{\gamma}}\right)^{q-1}\int_{B_{\rr}}\vp(x,\snr{D\eta}\tilde{v}) \tilde{v}^{-\tilde{\gamma}-1} \ \dx,
\end{flalign}
for $c=c(c_1,p,q)$. Absorbing terms in the previous inequality and setting $\gamma:=\tilde{\gamma}+1$, we obtain the announced inequality.
\end{proof}

For our main result, we refine the previous Caccioppoli estimate to involve oscillation of the supersolution.
\begin{lemma} \label{lem-cacc-osc}
Let $\vp \in \Phi_c(\Omega)$ satisfy (A0), (A1),  (aInc)$_p$ and (aDec)$_q$ with some $1<p\leq q<\infty$. Under assumptions \eqref{A-Car}--\eqref{A-monotone}, let $B_{2\rr} \Subset \Omega$ be any ball and  $v \in W^{1,\vp(\cdot)}(\Omega)$ be a~supersolution to \eqref{A0}, which is non-negative in $B_{2\rr}$. 
Then
\begin{align*}
\int_{B_{\rr}} \vp(x,\snr{D v}) \ \dx \leq c \int_{B_{2\rr}} \vp \left(x, \frac{\osc_{x \in B_{2\rr}} v (x)}{\rr}\right) \ \dx,
\end{align*}
where $c=c(c_1,p,q)$.
\end{lemma}

\begin{proof}Let $\eta \in C_c^{1}(B_{2\rr})$ be a cut-off function such that \[\chi_{B_{\rr}} \leq \eta \leq \chi_{B_{2\rr}}\quad\text{ and }\quad |D\eta|\leq \frac{2}{\rr}.\]

 Let us apply Lemma \ref{lem-cacc} for $\tilde v := v - \inf_{x \in B_{2\rr}} v$ and get
\begin{align*} 
\int_{B_{\rr}} \vp(x, |Dv|) \ \dx &= \int_{B_{\rr}} \tilde v^{\gamma} \tilde v^{-\gamma} \vp(x, |D\tilde v|) \ \dx \\
&\leq c (\osc_{x \in B_{2\rr}} v (x))^{\gamma} \int_{B_{2\rr}} \dfrac{\vp(x,|D\eta|\tilde v)}{\tilde v^{\gamma}} \ \ \dx.
\end{align*}
Now since $\vp$ satisfies (aInc)$_{p}$ and $\gamma \in (1,p)$, we see that $\dfrac{\vp(x,|D\eta|\tilde v)}{\tilde v^{\gamma}} \leq L_p \dfrac{\vp(x,|D\eta|\osc_{x \in B_{2\rr}} v (x))}{(\osc_{x \in B_{2\rr}} v (x))^{\gamma}}$. Therefore
\begin{align*}
\int_{B_{2\rr}} \vp(x, |Dv|) \ \dx &\leq c (\osc_{x \in B_{2\rr}} v (x))^{\gamma} \int_{B_{2\rr}} (\osc_{x \in B_{2\rr}} v (x))^{-\gamma} \vp \left(x,|D\eta| \osc_{x \in B_{2\rr}} v (x)\right) \ \dx \\
&= c \int_{B_{2\rr}} \vp \left((x,|D\eta| \osc_{x \in B_{2\rr}} v (x) \right) \ \dx\ \leq c\int_{B_{2\rr}} \vp\left(x,2\frac{ \osc_{x \in B_{2\rr}} v (x)}{\rr}\right) \ \dx,
\end{align*}
where we can get the claim by using doubling properties of $\vp$.
\end{proof}

We will use also the following classical iteration lemma in the proof of Lemma \ref{osc}.
\begin{lemma}\cite[Lemma~4.2]{hahato}\label{lem:iter}
Let $Z$ be a bounded non-negative function in the interval $[r, R]$ $\subset \mathbb{R}$ and let $X$ be a doubling function in $[0, \infty)$. Assume that there exists $\alpha \in [0,1)$ such that
\begin{align*}
    Z(t) \leq X \left( \frac{1}{s-t}\right) + \alpha Z(s)\quad\text{ for all $\quad r\leq t < s \leq R$.}
\end{align*}
 Then
\begin{align*}
    Z(r) \leq c\, X \left( \frac{1}{R-r}\right),
\end{align*}
where $c$ is a constant that depends only on the doubling constants of $X$ and $\alpha$.
\end{lemma}

In the next Lemma we show how to control the oscillation of a solution $v\in \mathcal{K}_{\psi}(\Omega)$ across the contact set via the oscillation of the obstacle $\psi$.
\begin{lemma}\label{osc} Suppose $A:\Omega\times\rn\to\rn$ satisfies \eqref{A-Car}--\eqref{A-monotone} with some
{$\vp \in \Phi_c(\Omega)$} satisfying assumptions (A0), (A1), (aInc)$_p$, (aDec)$_q$ with some $1<p\leq q\leq n$. Let $K\subset \Omega$ be a compact set and $v\in \mathcal{K}_{\psi}(\Omega)$ be a solution to problem \eqref{obs} with obstacle $\psi\in C(\Omega)$ such that
\begin{flalign}\label{43}
\snr{\psi(x_{1})-\psi(x_{2})}\le C_\psi\snr{x_{1}-x_{2}}^{\theta} \ \ \mbox{for all} \ \ x_{1}\in K,\  x_{2}\in \Omega,
\end{flalign}
where $\theta \in (0,1]$ and $C_\psi$ is a positive, absolute constant. Assume further that $\mu=-\diver\, A(x,Dv)$. Then, for  all $\bar{x}\in K$ and any $\rr\in \left(0,\frac{1}{40}\min\left\{1,\dist\{K,\partial\Omega\}\right\}\right)$ small enough for  $\rho_{{\vp(\cdot)};B_{18\rr}(\bar{x})}(|D v|) \le 1$, it holds 
\begin{flalign*}
\mu(B_{\rr}(\bar{x}))\le c(\data,{\psi})\, \rr^{-\theta}\int_{B_\rr(\bar{x})} \vp(x, \rr^{\theta-1})\, \dx.
\end{flalign*}
%where $\sigma:= 1-\tfrac{\beta_{0}}{q}(p-1)$.
\end{lemma}
\begin{proof}
Since $v\in \mathcal{K}_{\psi}(\Omega)$ is a solution to problem \eqref{obs} and $\psi \in C(\Omega)$, by Theorem \ref{T4}, the second part, $v$ is continuous. Given that $v$ is also a supersolution to \eqref{A0}, it realizes \eqref{sux}, so Riesz's representation theorem renders the existence of a unique, non-negative Radon measure $\mu$ such that for all $\eta \in C^{\infty}_{c}(\Omega)$ there holds
\begin{flalign}\label{meas}
\int_{\Omega}A(x,Dv)\cdot D\eta \ \dx -\int_{\Omega}\eta \ \d\mu=0 \ \ \mbox{in} \ \ \Omega.
\end{flalign}
Let us fix $\bar{x}\in K$ and $B_{\rr}(\bar{x})\subset \Omega$ with small $\rr$ to be fixed, such that $B_{16\rr}(\bar{x})\Subset \Omega$.
We define
\begin{align*}
\Omega_0: = \{x \in \Omega : v(x) > \psi(x)\}, \quad \Omega_c := \{x \in \Omega : v(x) = \psi(x)\}.
\end{align*}
If {$B_{\rr}(\bar{x})$ does not touch the contact set,} i.e. $B_{\rr}(\bar{x})\cap\Omega_{c}=\emptyset$, then, by Theorem \ref{T4}, the second part, the function $v$ is $\mathcal{A}_{\vp(\cdot)}$-harmonic in $B_{\rr}(\bar{x})$ and $\mu(B_{\rr}(\bar{x}))\equiv 0$. Hence, it suffices to consider only the case when $B_{\rr}\cap \Omega_{c}\not = \emptyset$. Let $x_{0}\in B_{\rr}(\bar{x})\cap\Omega_{c}$ and notice that, by monotonicity, $\mu(B_{\rr}(\bar{x}))\le \mu(B_{2\rr}(x_{0}))$. Note that by \eqref{43} it follows that
\begin{flalign}\label{42}
%\osc_{x\in B_{16\rr}(x_{0})}\psi(x)=& \left(\sup_{x\in B_{16\rr}(x_{0})}\psi(x)-\psi(\bar{x})\right)+ \left(\psi(\bar{x})-\psi(x_0)\right) \\\nonumber
%&+ \left(\psi(x_0)-\psi(\bar{x})\right)+ \left(\psi(\bar{x})-\inf_{x\in B_{16\rr}(x_{0})}\psi(x)\right) 
\osc_{x\in B_{16\rr}(x_{0})}\psi(x)=& \left(\sup_{x\in B_{16\rr}(x_{0})}\psi(x)-\psi(\bar x)\right) + \left(\psi(\bar x)-\inf_{x\in B_{16\rr}(x_{0})}\psi(x)\right) 
\le \  c(\theta,C_\psi)\rr^{\theta}. 
\end{flalign}
The final claim will be obtained by the use of Caccioppoli estimate from Lemma~\ref{lem-cacc-osc} involving $\osc_{x\in B_{4\rr}(x_{0})}v(x)$, which we need to estimate. We set
\begin{align*}
&\overline{v}(\rr) := \sup_{x \in B_\rr(x_0)} v(x), 
&\underline{v}(\rr):= \inf_{x \in B_\rr(x_0)} v(x), \qquad\qquad\\ 
&\overline{\psi}(\rr) := \sup_{x \in B_\rr(x_0)} \psi(x), & \underline{\psi}(\rr) := \inf_{x\in B_\rr(x_0)} \psi(x),\qquad \quad\ \ \\
&\vartheta_+:=\osc_{x\in B_{16\rr}(x_{0})}\psi(x)+\underline v(8\rr),   &\vartheta_{-}:=\osc_{x\in B_{16\rr}(x_{0})}\psi(x)-\underline v(8\rr)\ 
\end{align*}
and begin with noticing that for all $x\in B_{8\rr}(x_{0})$,
\begin{flalign}\label{44a}
\ v(x)-\vartheta_{+}&\le v(x)+
\osc_{x\in B_{16\rr}(x_{0})}\psi(x)-\inf_{x\in B_{8\rr}(x_{0})}v(x)=v(x)+\vartheta_{-},\\
\ v(x)+\vartheta_{-}&\ge 0.\label{44b}
\end{flalign} Additionally since $
\underline \psi(16\rr)\le {\underline \psi(8\rr)}\le \underline v(8\rr)\le v(x_{0})=\psi(x_{0})\le \overline \psi(16\rr),$
we infer that $ \nr{\psi}_{L^{\infty}(B_{16\rr}(x_{0}))}\le \vartheta_+$.  Moreover, due to~\eqref{42} and~\eqref{44a} it holds
\begin{align}\label{est}
    \sup_{x\in B_{4\rr}(x_{0})}(v(x)-\vartheta_{+})_{+} \leq c  \sup_{x\in B_{4\rr}(x_{0})}(v(x)-\underline{v}(8\rr))_{+} + c\,\rr^\theta.
\end{align}
Heading towards the estimate
\begin{flalign}\label{sup-inf}
\sup_{x\in B_{4\rr}(x_{0})}(v(x)-\vartheta_{+})_{+}&\leq  c \left( \eif_{B_{4\rr}(x_0)} (v(x)+\vartheta_{-}) +  \rr\right) + c\,\rr^\theta,
\end{flalign}
we show that for every $\sigma \in [\frac{1}{2},1)$ that
\begin{align}\label{sup-est}
     \sup_{x\in B_{8\sigma \rr}(x_{0})}(v(x)-\underline{v}(8\rr))_{+} &\leq \frac{c}{ (1-\sigma)^{4nq^2}} \left[ \left( \mint_{B_{8 \rr}(x_0)}  (v-\underline{v}(8 \rr))_+^q \ \dx \right)^{1/q}  + \rr \right] + c \rr^\theta.
\end{align}
This is almost the same as in \eqref{v-bounded}, but there is no average term. The lack of average term allows us to upgrade the exponent $q$ to any positive exponent $h$ later on. We proceed in two cases:  $\underline{v}(8 \rr) \geq \overline{\psi}(16\rr)$ or  $\underline{v}(8 \rr) \leq \overline{\psi}(16  \rr)$. 

\medskip

\noindent \textit{Case 1: $\underline{v}(8 \rr) \geq \overline{\psi}(16 \rr)$.} This case is simpler, as making use of \eqref{v-bounded} and choosing $\ell = \underline{v}(8 \rr)$ we get
\begin{align*}
    \sup_{x\in B_{8\sigma \rr}(x_{0})}(v(x)-\underline{v}(8  \rr))_{+} \leq \frac{c}{ (1-\sigma)^{4nq^2}} \left[ \left( \mint_{B_{8 \rr}(x_0)}  (v-\underline{v}(8 \rr))_+^q \ \dx \right)^{1/q} + \rr \right]
\end{align*}
and~\eqref{sup-est} follows. {Note that in this case \eqref{v-bounded} involves no average-term from  as $v-\underline{v}(8\rr)$ is non-negative in $B_{8\rr}$.}
\medskip

\noindent \textit{Case 2: $\underline{v}(8 \rr) \leq \overline{\psi}(16 \rr)$.} Here we use H\"older continuity of $\psi$ and the estimate \eqref{v-bounded} to see that 
\begin{align*}
    \sup_{x\in B_{8\sigma \rr}(x_{0})}(v(x)-\underline{v}(8 \rr))_{+} \leq \sup_{x\in B_{8\sigma \rr}(x_{0})}(v(x)-\underline{\psi}(16 \rr))_{+} \leq \sup_{x\in B_{8\sigma \rr}(x_{0})}(v(x)-\overline{\psi}(16 \rr))_{+} + c \rr^\theta \\
    \leq \frac{c}{ (1-\sigma)^{4nq^2}}\left[ \left( \mint_{B_{8 \rr}(x_0)}  (v-\overline{\psi}(16 \rr))_+^q \ \dx \right)^{1/q} + |(v-\overline{\psi}(16 \rr))_{B_{8 \rr}(x_0)}| +  \rr \right] + c \rr^\theta.
\end{align*}
Since in this case $\underline{v}(8\rr) \leq \overline{\psi}(16 \rr)$ we continue to estimate {the integral}
\begin{align}\label{sup-with-average}
\begin{split}
    \sup_{x\in B_{8\sigma \rr}(x_{0})}(v(x)-\underline{v}(8 \rr))_{+} \leq  \frac{c}{ (1-\sigma)^{4nq^2}} \bigg[ &\left( \mint_{B_{8 \rr}(x_0)}  (v-\underline{v}(8 \rr))_+^q \ \dx \right)^{1/q} \\ &+ |(v-\overline{\psi}(16 \rr))_{B_{8 \rr}(x_0)}| + \rr \bigg] + c \rr^\theta.
\end{split}
\end{align}
Now we focus on the average term. %If it did not exist in the previous estimate, we could replace the exponent $q$ by any positive exponent $h$. 
By the H\"older continuity of $\psi$ we get
\begin{align*}
     |(v-\overline{\psi}(16 \rr))_{B_{8 \rr}(x_0)}| &= \left| \mint_{B_{8 \rr}(x_0)} v- \underline{\psi}(16 \rr) + \underline{\psi}(16 \rr) - \overline{\psi}(16 \rr) \, \dx \right | \\
     &\leq \mint_{B_{8 \rr}(x_0)} v - \underline{\psi}(16 \rr) \ \dx + c \rr^\theta \leq \mint_{B_{8 \rr}(x_0)} v - \overline{\psi}(16 \rr) \ \dx + c \rr^\theta \\
     &\leq \mint_{B_{8 \rr}(x_0)} v - \underline{v}(8 \rr) \ \dx + c  \rr^\theta.
\end{align*}
Since by assumption $q>1$, we can use H\"older's inequality to increase the exponent and arrive to
\begin{align*}
    |(v-\overline{\psi}(16 \rr))_{B_{8 \rr}(x_0)}| \leq \left (\mint_{B_{8 \rr}(x_0)} (v - \underline{v}(8 \rr))_+^q \ \dx\right)^{1/q} + c  \rr^\theta.
\end{align*}
This estimate allows us to absorb the average term to the integral term in \eqref{sup-with-average} so we can conclude with~\eqref{sup-est} in both cases.

\medskip

\noindent \textit{Upgrading inequality~\eqref{sup-est} to have instead of $q$ any $h \in (0,\infty)$.} Since $\vp(\cdot)$ satisfies {also} (aDec)$_{q_0}$ for any $q_0>q$, we can assume that $q>h$. Note that for the previous inequality we do not need the assumption that $q \leq n$. Let $\sigma,\tau>0$ be constants that satisfy $4\rr \leq 8 \sigma \rr <  8\tau \rr \leq 8\rr$ and denote \[Z(\sigma) := \eup_{B_{8\sigma\rr}(x_0)} (v-\underline{v}(8 \rr))_+.\] {Restating \eqref{sup-est} with this notation we get} and recalling that $\tau \leq \rr \leq 1$
\begin{align*}
    Z(\sigma) &\leq c \left(1-\frac{\sigma}{{\tau}} \right)^{-4nq^2} \left( \mint_{B_{8\tau \rr}(x_0)} (v-\underline{v}(8 \rr))_+^q \ \dx \right)^{1/q}  + \tilde c(\psi, \rr, \theta) \\
    & \leq c\left( \dfrac{1}{\tau-\sigma}\right)^{4nq^2} \left( \mint_{B_{8\tau \rr}(x_0)} (v-\underline{v}(8 \rr))_+^q \ \dx \right)^{1/q}  + \tilde c(\psi, \rr, \theta).
\end{align*}
Since $8\tau \rr \in (4\rr,8\rr)$, we see that
\begin{align*}
    \left( \mint_{B_{8\tau \rr}(x_0)} (v-\underline{v}(8 \rr))_+^q \ \dx \right)^{1/q} &\leq c \left(\mint_{B_{8\rr}(x_0)} (v-\underline{v}(8 \rr))_+^h Z(\tau)^{q-h} \ \dx \right)^{1/q} \\
    &\leq c\left( \mint_{B_{8 \rr}(x_0)} (v-\underline{v}(8 \rr))_+^h \ \dx \right)^{1/q} Z(\tau)^\frac{q-h}{q}.
\end{align*}
Now Young's inequality with exponents $\tfrac qh$ and $\frac{q}{q-h}$ yields
\begin{align*}
    Z(\sigma) &\leq c \left(\dfrac{1}{\tau-\sigma}\right)^{4nq^2} \left( \mint_{B_{8 \rr}(x_0)} (v-\underline{v}(8 \rr))_+^h \ \dx \right)^{1/q} Z(\tau)^{\frac{q-h}{q}} + \tilde c(\psi, \rr, \theta) \\
    &\leq \dfrac{ch}{q}\left(\dfrac{1}{\tau-\sigma}\right)^\frac{4nq^3}{h} \left(\mint_{B_{8 \rr}(x_0)} (v-\underline{v}(8\rr))_+^h \ \dx \right)^{1/h} + \tilde c(\psi, \rr, \theta) + \frac{q-h}{q}Z(\tau).
\end{align*}
Denoting the right-hand side without the term $\frac{q-h}{q}Z(\tau)$ as $X\left(\frac{1}{\tau-\sigma}\right)$, we have arrived at the starting point of a classical iteration lemma (Lemma~\ref{lem:iter})
\begin{align*}
    Z(\sigma) \leq X\left(\frac{1}{\tau-\sigma}\right) + \frac{q-h}{q}Z(\tau),
\end{align*}
where $Z$ is bounded, $X$ is doubling and $\frac{q-h}{q} \in (0,1)$.
Performing the iteration we end up with the desired estimate
\begin{align*}
    \sup_{x\in B_{4\rr}(x_{0})}(v(x)-\underline{v}(8\rr))_{+}&\le c\left(\mint_{B_{8\rr(x_0)}}(v(x)-\underline{v}(8\rr))_+^{h} \ \dx\right)^{{\frac{1}{h}}}+ c\rr +c\rr^\theta\quad\text{
for all $h >0$.}
\end{align*}

\medskip

\noindent \textit{The final estimate on $\sup_{x\in B_{4\rr}(x_{0})}(v(x)-\vartheta_{+})_{+}$.} Using~\eqref{est}, choosing $h=h_0$ from~\eqref{32} we may combine the content of the previous display with~\eqref{44a} to get 
\begin{flalign*}
\sup_{x\in B_{4\rr}(x_{0})}(v(x)-\vartheta_{+})_{+}&\leq \sup_{x\in B_{4\rr}(x_{0})}(v(x)-\underline{v}(8\rr))_{+} + c \rr^\theta \\
&\le c\left(\mint_{B_{8\rr(x_0)}}(v(x)-\underline{v}(8\rr))_+^{h_0} \ \dx\right)^{\frac{1}{h_0}}+ c\rr +c\rr^\theta\\
&\le c\left(\mint_{B_{8\rr(x_0)}}(v(x)+\vartheta_{-})^{h_0} \ \dx\right)^{\frac{1}{h_0}}+ c\rr +c\rr^\theta\\
&\le c \left( \eif_{B_{4\rr}(x_0)} (v(x)+\vartheta_{-}) +  \rr\right) + c\rr^\theta,
\end{flalign*}
where {\it $c=c(\data,{\psi})$} and $h_{0}$ is the exponent appearing in \eqref{32}. As $\rr <1$ and $\theta <1$, it follows that $\rr \leq \rr^\theta$ and we get~\eqref{sup-inf}. 

\medskip

\noindent \textit{Estimate on $\osc_{x\in B_{4\rr}(x_{0})}v(x)$.} 
From $\eqref{44b}$, we see that $v+\vartheta_{-}$ is a non-negative supersolution to \eqref{A0} in $B_{8\rr}(x_{0})$, thus, using the definition of $\vartheta_{+}$,~\eqref{sup-inf}, and recalling that $v(x_0)=\psi(x_0)$ we have 
\begin{align*}
\osc_{x\in B_{4\rr}(x_{0})}v(x)&=\osc_{x\in B_{16\rr}(x_{0})}\psi(x)+\sup_{x\in B_{4\rr}(x_{0})}v(x)-\inf_{x\in B_{4\rr}(x_{0})}v(x)-\osc_{x\in B_{16\rr}(x_{0})}\psi(x)\\
&\le\osc_{x\in B_{16\rr}(x_{0})}\psi(x)+\sup_{x\in B_{4\rr}(x_{0})}(v(x)-\vartheta_{+})_{+}\\
&\le\osc_{x\in B_{16\rr}(x_{0})}\psi(x)+ c\inf_{x\in B_{4\rr}(x_{0})}(v(x)+\vartheta_{-})+c\rr^\theta \\
&\le c\left(\osc_{x\in B_{16\rr}(x_{0})}\psi(x)+\inf_{x\in B_{4\rr}(x_{0})}\left( v(x)+\osc_{x\in B_{16\rr}(x_{0})}\psi(x)-\inf_{x\in B_{8\rr}(x_{0})} v(x)\right)\right)+c\rr^\theta\\
&\le c \left(\osc_{x\in B_{16\rr}(x_{0})}\psi(x)+   \psi(x_0)-\inf_{x\in B_{16\rr}(x_{0})} \psi(x)\right)+c\rr^\theta\\ 
&\le c\osc_{x\in B_{16\rr}(x_{0})}\psi(x)+c\rr^\theta,
\end{align*}

which due to \eqref{42} implies that
\begin{flalign}\label{45}
\osc_{x\in B_{4\rr}(x_{0})}v(x) %\le c\osc_{x\in B_{8\rr}(x_{0})}\psi(x)+ c\rr
\le c\rr^{\theta}\qquad\text{with }\ c=c(\data,\psi).
\end{flalign}

\medskip

\noindent \textit{Conclusion. } Now, set 
 \begin{equation}
     \label{v-tilde}\tilde{v}:=v-\inf_{x\in B_{4\rr}(x_{0})}v(x),
 \end{equation} notice that $D\tilde{v}=Dv$ and pick $\eta\in C^{1}_{c}(B_{4\rr}(x_{0}))$ such that $\chi_{B_{2\rr}(x_{0})}\le \eta\le \chi_{B_{4\rr}(x_{0})}$ and $\snr{D\eta}\le \rr^{-1}$. Clearly, recalling also Theorem \ref{T4}, $\tilde{v}$ is a bounded supersolution to \eqref{A0} which is non-negative in $B_{4\rr}(x_{0})$, thus, by Young inequality, we obtain
%and the fact that $(q-1)p/(p-1)\ge q$,
we obtain  
\begin{flalign}\label{46}
\mu(B_{2\rr}(x_{0}))\le& \int_{B_{4\rr}(x_{0})}\eta^{q} \ \d\mu=q\int_{B_{4\rr}(x_{0})}\eta^{q-1}A(x,Dv)\cdot D\eta \ \dx\nonumber\\
&\leq c_2q \int_{B_{4\rr}(x_0)}\dfrac{1}{\rr^\theta} \dfrac{\vp(x, |Dv|)}{ |Dv|} \rr^\theta|D\eta| \ \dx\nonumber\\
& \leq c_2q \int_{B_{4\rr}(x_0)}\dfrac{1}{\rr^\theta} \vp^{\ast}\left (x,\dfrac{\vp(x, |Dv|)}{|Dv|} \right ) + \dfrac{1}{\rr^\theta}\vp(x, \rr^\theta |D\eta|) \ \dx.
\end{flalign}
Since $\vp$ is convex, we can estimate further as follows 
\begin{align*}
\mu(B_{2\rr}(x_0)) \leq  c \int_{B_{4\rr}(x_0)} \dfrac{1}{\rr^\theta} \vp(x, |Dv|) + \dfrac{1}{\rr^\theta} \vp(x,\rr^\theta |D \eta|) \ \dx.
\end{align*} 
Now, for the first term on the right-hand side we use Caccioppoli inequality  from Lemma~\ref{lem-cacc-osc} and oscillation estimate \eqref{45}, for the second term we use the cut-off estimate for $|D\eta|$. This yields the same estimates for each of the terms from the previous display. Namely, we have 
\begin{align*}
\mu(B_{2\rr}(x_0)) &\leq  c \int_{B_{4\rr}(x_0)} \dfrac{1}{\rr^\theta} \vp\left(x,    \frac{\osc_{x \in B_{4 \rr}(x_0)} v(x)}{\rr}\right) + \dfrac{1}{\rr^\theta} \vp\left(x, \rr^{\theta-1}\right) \ \dx \\
& \leq  2 c \rr^{-\theta}\int_{B_{8\rr}(x_0)}   \vp\left(x, \rr^{\theta-1}\right)   \dx.
\end{align*}
Since $\vp$ satisfies assumptions in \eqref{A1-n} and the fact that $B_{\rr}(\bar{x}) \subset B_{2\rr}(x_0)$ allow us to change the domain of integration
\begin{align*}
\rr^{-\theta}\int_{B_{8\rr}(x_0)} \vp(x, \rr^{\theta-1}) \ \dx &\leq c \rr^{n-\theta} \vp_{B_{8\rr}(x_0)}^{-}(\rr^{\theta-1})\\
& \leq c \rr^{n-\theta} \vp^{-}_{B_{\rr}(\bar{x})}(\rr^{\theta-1}) \leq c \rr^{-\theta}\int_{B_{\rr}(\bar{x})} \vp\left(x, \rr^{\theta-1}\right) \ \dx.
\end{align*}
Previous two displays combined finishes the proof. 
\end{proof}

\subsection{Proof of Theorem \ref{T6}} We prove here our main result on the removability of singularities for solutions to \eqref{A0}. Without loss of generality we suppose $U\Subset \Omega$ is an open set with $U\cap E \not =\emptyset$, because otherwise the measure vanishes and there is nothing to prove. We will show that when $U\cap E \not =\emptyset$, we also have $\mu(E\cap U)=0$ and $\mu(U\setminus E)=0$.

Let $v\in \mathcal{K}_{u}(U)$ be the unique solution to problem \eqref{obs} in $U$ under the conditions of Proposition~\ref{prop:ex-uni}. As noticed in the beginning of the proof of {Lemma~\ref{osc}}, we infer by Riesz's representation theorem that $\mu=-\diver\, A(x,Dv)$ is a non-negative Radon measure. Fix a compact set $K\Subset E\cap U$. For any $x_{0}\in K$ and all $\rr\in \left(0,\frac{1}{40}\min\left\{1,\dist\{K,\partial (E\cap U)\}\right\}\right)$ so small  that $\rho_{{\vp(\cdot)};B_{18\rr} (x_0)}(|D v|) \leq 1$, Lemma \ref{osc} yields  the following decay estimate 
\begin{flalign}\label{51}
\mu(B_{\rr}(x_{0}))\le c\rr^{-\theta}  \int_{B_{\rr}(x_0)} \vp(x, \rr^{\theta-1}) \ \dx\quad \text{with }c=c(\texttt{data},{u}).
\end{flalign}

We assume that $\mathcal{H}_{\cJp}(E)=0$,  so also  $\mathcal{H}_{\cJp}(K)=0$. Therefore, for any $\varepsilon>0$ the set $K$ can be covered with balls $B_{\rr_{j}}(x_{j})$ with radii $\rr_{j}$ less than $\frac{1}{80}\min\left\{1,\dist\{K,\partial (E\cap U)\}\right\}$ and $(x_{j})_{j \in \mathbb{N}}\subset K$, so that
\begin{flalign*}
\mu(K)\le \sum_{j=1}^{\infty}\mu(B_{\rr_{j}}(x_{j}))\stackrel{\eqref{51}}{\le} c\sum_{j=1}^{\infty} \rr^{-\theta}\int_{B_{\rr_{j}}(x_{j})}\vp\left(x,{\rr}^{\theta-1}\right) \ \dx \le \varepsilon.
\end{flalign*}
Since $\varepsilon$ is arbitrary, we conclude that $\mu(K)=0$. Furthermore, since $\mu$ is a~Radon measure, $K$ is an arbitrary compact subset of $E\cap U$ and $E\cap U$ is $\mu$-measurable, also $\mu(E\cap U)=0$.

In order to prove that $\mu(U\setminus E)=0$ we take $\eta\in W^{1,\vp(\cdot)}_{0}(U\setminus E)$, $\eta(x)\ge 0$ for a.e. $x\in U\setminus E$ and, for $s>0$, set $\eta_{s}:=\min\left\{s\eta,v-u\right\}$. Then $\eta_{s}\in W^{1,\vp(\cdot)}_{0}((U\setminus E)\cap \{x\colon v(x)>u(x)\})$. Consequently, by the second claim of Theorem \ref{T4} we get that
\begin{flalign}\label{x1} 
\int_{U\setminus E}A(x,Dv)\cdot D\eta_{s} \ \dx=\int_{(U\setminus E)\cap \{v>u\}}A(x,Dv)\cdot D\eta_{s} \ \dx=0.
\end{flalign}
On the other hand, since $u$ solves \eqref{A0} in $U\setminus E$, then
\begin{flalign}\label{x2}
\int_{U\setminus E}A(x,Du)\cdot D\eta_{s} \ \dx =0.
\end{flalign}
We subtract \eqref{x2} from \eqref{x1} and obtain
\begin{flalign}\label{x3}
\int_{U\setminus E}\left(A(x,Dv)-A(x,Du)\right)\cdot D\eta_{s} \ \dx=0.
\end{flalign}
Since the operator is monotone, from \eqref{x1}-\eqref{x3} we infer that
\begin{flalign*}
\int_{(U\setminus E)\cap\{s\eta\le v-u\}}&\left(A(x,Dv)-A(x,Du)\right)\cdot D\eta \ \dx\nonumber \\
&\le -\frac{1}{s}\int_{(U\setminus E)\cap\{ s\eta>v-u\}}\left(A(x,Dv)-A(x,Du)\right)\cdot (Dv-Du) \ \dx\le 0.
\end{flalign*}
By the Lebesgue's dominated convergence theorem we conclude that 
\begin{flalign*}
\int_{U\setminus E}\left(A(x,Dv)-A(x,Du)\right)\cdot D\eta \ \dx\le 0.
\end{flalign*}
Since $u$ is a solution to \eqref{A0}, we obtain
\begin{flalign*}
\int_{U\setminus E}A(x,Dv)\cdot D\eta \ \dx \le 0\quad\text{for all non-negative $\eta\in W^{1,\vp(\cdot)}_{0}(U\setminus E)$.} 
\end{flalign*}
Therefore, recalling \eqref{meas}, $\mu(U\setminus E)=0$. Hence, $\mu(U)=0$, which means
\begin{flalign}\label{60}
\int_{U}A(x,Dv)\cdot Dw \ \dx =0 \ \ \mbox{for all} \ \ w\in W^{1,\vp(\cdot)}_{0}(U).
\end{flalign}
Let us set $\hat{A}(x,z):=-A(x,-z)$ and notice that  $\hat{A}$ satisfies \eqref{A}. We consider $\hat{v}\in W^{1,\vp(\cdot)}(U)$ being a solution to the obstacle problem
\begin{flalign*}
\int_{\Omega}\hat{A}(x,D\hat{v})\cdot (Dw-D\hat{v}) \ \dx\ge 0 \quad\text{ for all }\ w\in \mathcal{K}_{-u}(U).
\end{flalign*}
By the arguments as above $\hat{v}$ satisfies
\begin{flalign}\label{61}
\int_{U}-A(x,-D\hat{v})\cdot Dw \ \dx=\int_{U}\hat{A}(x,D\hat{v})\cdot Dw \ \dx=0\quad\text{
for all $\ w\in W^{1,\vp(\cdot)}_{0}(U)$.}
\end{flalign} 

What remains to show is that $v=-\hat{v}=u$. We define $\bar{w}:=v+\hat{v} \in W^{1,\vp(\cdot)}(U)$ and use it as a test function for \eqref{60} and \eqref{61}. Adding the resultant equations we obtain
\begin{flalign*}
0\le  \int_{U}\left(A(x,Dv)-A(x,-D\hat{v})\right)\cdot(Dv+D\hat{v}) \ \dx =0.
\end{flalign*} due to the monotonicity of the operator. Consequently, $D(v+\hat{v})=0$ a.e. in $U$. Since $v+\hat{v}\in W^{1,\vp(\cdot)}_{0}(U)$, we have $v=-\hat{v}$ a.e. in $U$ and, by {Theorem~\ref{T4}}, $-\hat{v}\le u\le v$ a.e. in $U$. From \eqref{60} and \eqref{61} it results that $u$ solves \eqref{A0} in $U$. 

Since $U$ is arbitrary, we can conclude that $u$ solves \eqref{A0} in $\Omega$, thus $E$ is removable.\qed

\end{document}